\documentclass[12pt,a4paper]{article}
\usepackage[dvipdfm]{hyperref}
\usepackage{amsmath,amsthm,amssymb,bbm,color,enumerate,mathrsfs,natbib}

\def\A{\mathcal{A}}

\def\C{\mathcal{C}}
\def\E{\mathcal{E}}
\def\F{\mathcal{F}}

\def\K{{\mathcal{K}}}

\def\M{\mathcal{M}}
\def\N{\mathbb{N}}

\def\R{\mathbb{R}}

\def\U{\mathcal{U}}
\def\W{{\mathcal{W}}}

\def\Z{{\mathbb{Z}}}

\theoremstyle{plain}
\newtheorem{thm}{Theorem}[section]

\newtheorem{lem}{Lemma}[section]
\newtheorem{prop}{Proposition}[section]

\theoremstyle{definition}
\newtheorem{assmp}{Assumption}[section]
\newtheorem{dfn}{Definition}[section]
\newtheorem{exmp}{Example}[section]
\newtheorem{obsv}{Observation}[section]
\newtheorem{rem}{Remark}[section]

\numberwithin{equation}{section}
\allowdisplaybreaks

\title{Relaxed Large Economies with Infinite-\hspace{0pt}Dimensional Commodity Spaces: The Existence of Walrasian Equilibria\thanks{This research was conducted while Khan held the position of Visiting Research Fellow at the Australian National University, February 15--April 15, 2016, and he thanks the department for its hospitality, and Patrick Beissner and Bogdan Klishchuk for stimulating conversation. A preliminary version of the paper was presented at the ``Fifth Asian Conference on Nonlinear Analysis and Optimization'',  Niigata, Japan on August 4, 2016. The final  version was completed during Sagara's visit to Johns Hopkins University, August 13--21, 2016; it has benefitted from the careful reading and helpful comments of two anonymous referees. This research is supported by JSPS KAKENHI Grant No.\,26380246 from the Ministry of Education, Culture, Sports, Science and Technology, Japan.}}
\date{\today}
\author{M. Ali Khan \\
{\small Department of Economics, The Johns Hopkins University} \\
{\small Baltimore, MD 21218, United States} \\
{\small e-mail: akhan@jhu.edu}
\and \\
Nobusumi Sagara\thanks{Corresponding author.} \\
{\small Department of Economics, Hosei University} \\
{\small 4342, Aihara, Machida, Tokyo 194--0298, Japan} \\
{\small e-mail: nsagara@hosei.ac.jp}}
\begin{document}
\maketitle
\setcounter{page}{0}
\thispagestyle{empty}
\clearpage

\begin{abstract} 
\noindent
Whereas \lq\lq convexification by aggregation'' is a well-understood procedure in mathematical economics,  \lq\lq convexification by
 randomization" has largely been limited to theories of statistical decision-making, optimal control and  non-cooperative games.  In this paper, in the context of classical Walrasian general equilibrium theory, we offer a comprehensive treatment of {\it relaxed economies} and their {\it relaxed Walrasian equilibria}:
our results pertain to a setting  with a  finite or 
a continuum of agents, and a continuum of commodities modeled either as an ordered separable Banach space or as an $L^\infty$-space. 
As a substantive consequence, we demonstrate that the convexity hypothesis
 can be removed from the original large economy under the saturation hypothesis, and that existing results in the antecedent literature can be 
effortlessly  recovered.
\end{abstract}


\tableofcontents

\section{Introduction}
Large economies were  introduced by \citet{au66}  as a prototype of perfect competition, and he demonstrated the existence of a Walrasian equilibrium  in the setting of a finite-dimensional commodity space  and a continuum of agents modeled as a nonatomic finite measure space. In his theorem, Aumann dispensed with the assumption of  convexity of preferences: an insight rather   remarkable for its time.  Since then, there have been several attempts to extend the theorem  to a setting with an infinite-dimensional commodity space, and they have all stumbled on the  well-known failure of the Lyapunov convexity theorem in infinite dimensions, and thereby failed to subdue the possible nonconvexity of the aggregate demand set. This is the reason why convexity assumptions on preferences are pervasive and inevitable even under the nonatomicity hypothesis on the set of agents,  as in  \cite{ky91,no97} for ordered separable Banach spaces,  and in  \cite{be91} for $l^\infty.$  This has led to efforts to strengthen the hypothesis along the intuition that there be ``many more agents than commodities":  in addition to    \cite{mdr03,ry91}, see  \cite{po97} for its formalization of the intuition as a condition on the nonatomic disintegration of the population measure of agents. 

However, recent work has established the  validity of  the Lyapunov convexity theorem, and as its corollary, the convexity property of the integral of a multifunction,    under the  reasonable assumption that the underlying  measure space of agents is  a {\it saturated} space. Both  separable Banach spaces and the  dual of separable Banach spaces have been considered, and indeed, this saturation property has been shown to be both  necessary and sufficient for the results; see \cite{ks13,ks15,ks16,po08,sy08}. It is thus natural then that one looks  for a generalization of Aumann's theorem under the assumption of a saturated measure space of agents by exploiting  the  (exact) convexity of the 
aggregate  demand set in the setting of an infinite-dimensional commodity space even   when the individual demand sets are not convex.  Such a result has now been satisfactorily executed by  \citet{le13}, and he also handles the difficulty arising from the failure of the joint continuity of the valuation functional.\footnote{As is by now well-known and well-understood, this was the third mathematical difficulty emphasized by  \citet[pp.\,224--225]{be91} when \lq\lq there are both infinitely many commodities and infinitely many participants in an economy.''  The other two difficulties that were identified concerned the facts that  (i)  \lq\lq there does not exist an infinite dimensional version of Fatou's lemma", and that (ii) \lq\lq budget sets in $L^\infty$ are not typically norm bounded and hence weak-star compact, even when they are defined by price systems in $L^1$.'' For a comprehensive discussion of these issues in the context of economies with a finite set of agents, as well as additional difficulties arising from the non-emptiness of the interior of the positive cone of a commodity space, see \citet{mz91}. \label{fn:diff} } However, it bears emphasis that  the result in \cite{le13} does not apply to economies  with a finite set of agents, but only to those with a continuum.\footnote{To be precise, the result applies to economies with atoms only under the convexity hypothesis on their preferences.}

In this paper, we take an approach alternative  to \lq\lq convexification by aggregation", by drawing on \lq\lq convexification  by randomization'',  a completely different operation from aggregation, and one that is also valid for economies with a finite set of agents. This procedure, and its terminology of \lq\lq relaxed controls", is  well-established in optimal control theory, and explored in, by now classical, work  of \citet{mc67,wa72,yo69}. To be sure, this  is  not an altogether novel direction, and it has been pursued with varying emphases in \citet{pt84}, \citet{no92}, and \citet{ba08}.\footnote{\citet{pt84} have incentive compatibility in economies with a finite set of agents as  their basic thrust; \citet{no92} is concerned with correlated equilibria; and \citet{ba08} is after a synthesis of  the Nash mixed-strategy theorem in games with a continuum of agents and of 
 general equilibrium theory with externalities and price-dependent preferences. He assumes, as he must, the hypotheses of convexity or single-valuedness of individual demands; but see Footnote \ref{fn1} below. \label{fn0}}   This work takes up the notion of  ``relaxation'' of the relevant solution concept by focusing on a randomized choice by the optimizing agent.  As specifically illustrated in \cite{pt84}, this randomization device is recognized to be more or less artificial, but that it copes with the nonconvex constraints which stem from idiosyncratic shocks to each agent.\footnote{Contrary to \cite{cs83}, randomness under consideration is different from extrinsic uncertainty that is unrelated to preferences and endowments of an economy, which results in state-dependent equilibria under the convexity hypothesis.} In this paper, our sole concern is Walrasian general equilibrium theory, and more specifically, on Aumann's theorem with non-convex preferences, and we  incorporate the refinements of the relaxation technique as offered in \cite{sa16}. 

The procedure for the relaxation of economies that is pursued in this paper is as follows.
\begin{itemize}
\item The preferences of each agent possesses a utility representation on a common consumption set $X$ formalized as a 
Polish subset of a Banach space. 
\item Utility functions on $X$ are then extended to the set $\Pi(X)$ of probability measures on $X$, and are thereby  an affine extension on $\Pi(X)$. Extended preferences on $\Pi(X)$ are consistent with the expected utility hypothesis. 
\item Each probability measure in $\Pi(X)$ is regarded as a randomized commodity (a lottery) over $X$. Given a market price, each agent can purchase the barycenter of a probability measure, which is a convex combination of commodities in $X$ with respect to a probability measure in $\Pi(X)$ under the budget constraint.   
\item Barycentric commodities under the budget constraint, so to speak, are evaluated in terms of  expected utilities and constitute the relaxed demand set of each agent. It  is thereby defined as a closed convex subset of $\Pi(X)$.\footnote{Since the preferences of each agent are not assumed to be convex, it should be noted that the individual demand set for the original economy lacks  convexity.} 
\item Relaxed allocations are well-defined in a way that the aggregate of the barycentric commodities of each agent does not exceed the total endowment. Relaxed Walrasian equilibria for relaxed economies are thereby formulated in a consistent manner. 
\item  Dirac measures in $\Pi(X)$  reduce to the usual notion of a commodity in $X,$ and consequently, Walrasian equilibria for the original economy are identified with relaxed (purified) Walrasian equilibria for the relaxed economy. 
\end{itemize}

We then offer for classical  Walrasian general equilibrium theory a supplementation to the usual convexification  method: under the saturation hypothesis, we can always construct a Walrasian equilibrium for the original economy from a relaxed Walrasian equilibrium for the relaxed economy. The point is that the existence of the latter is easily established on account of the fact that the relaxed economy is {\it already} a convexification of the original economy. This is done through a \lq\lq purification principle", as in  \cite{ks14b,sa16},  a powerful tool whose utility in 
non-cooperative game theory and statistical decision theory is well-understood.       As to the difficulty of  the joint continuity of the valuation functional for  price-commodity pairs in infinite dimensions when one attempts to apply the fixed point theorem,\footnote{We note the use of   the infinite-dimensional version of the Gale--Nikaido lemma and Michael's selection theorem in  \citet{ya85}. We also note that this difficulty of joint continuity of the valuation functional does not arise when one is working in a finite-dimensional space. Thus, in \cite{ba08}, even though the commodity space of probability measures on the finite-dimensional commodity space is infinite dimensional, the valuation functional is jointly continuous by virtue of the price space being a finite-dimensional Euclidean space.}  we adapt the technique employed in \cite{le13,po97} to relaxed demand sets. In short summary, we can remove the convexity hypothesis  from \cite{ky91}, and recover   the existence result in \cite{le13,po97,ry91} under the saturation hypothesis, all  in the context of a separable Banach space. Furthermore, in $L^\infty$-spaces, we can also remove the convexity hypothesis from \cite{be91} and  derive the existence of Walrasian equilibria with free disposal under the saturation hypothesis. 
 
As applications, we offer two examples. We illustrate how our existence result yields the existence of Pareto optimal, envy-free allocations in large economies with infinite-dimensional commodity spaces; this sharpens the classical result of \cite{va74}. Curiously, envy-freeness is naturally interpreted as incentive compatibility in random economies where each agent incurs an idiosyncratic shock that  characterizes his/her type along the lines of \cite{pt84}. We also demonstrate that our existence result is valid for economies with indivisible commodities along the lines of \cite{ky81}; this presents an alternative approach to the existence result on economies with indivisible commodities investigated in \cite{dkm01} and \cite{sv84} via the alternative convexification technique.

\section{Preliminaries}  We develop four  ideas in three subsections,: relaxed controls in optimal control theory, the purification principle now also pervasive in application to the theory of non-atomic games, the saturation property and Gelfand integration of functions taking values in the space of essentially bounded measurable functions.

\subsection{Relaxed Controls}
We denote by $\Pi(X)$ the set of probability measures on a Polish space $X$ furnished with the Borel $\sigma$-algebra $\mathrm{Borel}(X)$. We endow $\Pi(X)$ with the \textit{topology of weak convergence} of probability measures, which is the coarsest topology on $\Pi(X)$ for which the integral functional $P\mapsto \int v dP$ on $\Pi(X)$ is continuous for every bounded continuous function $v:X\to \R$. Then $\Pi(X)$ is also a Polish space; see \citet[Theorem 15.15]{ab06}. Let $(T,\Sigma,\mu)$ be a finite measure space. (Throughout the paper, we always assume that it is complete.) By $\M(T,X)$ we denote the space of measurable functions from $T$ to $X$ and by $\mathcal{R}(T,X)$ the space of measurable functions from $T$ to $\Pi(X)$. Each element in $\M(T,X)$ is called a \textit{control} and that in $\mathcal{R}(T,X)$ is called a \textit{relaxed control} (a \textit{Young measure}, a \textit{stochastic kernel}, or a \textit{transition probability}), which is a probability measure-valued control. For every function $\lambda:T\to \Pi(X)$, the real-\hspace{0pt}valued function $t\mapsto \lambda(t)(C)$ is measurable for every $C\in \mathrm{Borel}(X)$ if and only if $\lambda$ is measurable; see \citet[Lemma 2]{po09}. By $\Delta(X)$, we denote the set of Dirac measures on $X$, i.e., $\delta_x\in \Delta(X)$ whenever for every $C\in \mathrm{Borel}(X)$: $\delta_x(C)=1$ if $x\in C$ and $\delta_x(C)=0$ otherwise. Each control $f\in \M(T,X)$ is identified with the Dirac measure valued control $\delta_{f(\cdot)}\in \mathcal{R}(T,X)$ satisfying $\delta_{f(t)}\in \Delta(X)$ for every $t\in T$.

A real-valued function $u:T\times X\to \R$ is a \textit{Carath\'eodory function} if $t\mapsto u(t,x)$ is measurable for every $x\in X$ and $x\mapsto u(t,x)$ is continuous for every $t\in T$. The Carath\'eodory function $u$ is jointly measurable; see \citet[Lemma 4.51]{ab06}. A Carath\'eodory function $u$ is \textit{integrably bounded} if there exists $\varphi\in L^1(\mu)$ such that $|u(t,x)|\le \varphi(t)$ for every $(t,x)\in T\times X$. Denote by $\C^1(T\times X,\mu)$ the space of integrably bounded Carath\'eodory functions on $T\times X$. For each $u\in \C^1(T\times X,\mu)$, define the integral functional $I_u:\mathcal{R}(T,X)\to \R$ by $I_u(\lambda)=\iint u(t,x)\lambda(t,dx)d\mu$. The \textit{weak topology} on $\mathcal{R}(T,X)$ is defined as the coarsest topology for which every integral functionals $I_u$ is continuous for every $u\in \C^1(T\times X,\mu)$. If $T$ is a singleton, then the set $\mathcal{R}(T,X)$ coincides with the set $\Pi(X)$. In this case $\C^1(T\times X,\mu)$ coincides with the space $C_b(X)$ of bounded continuous functions on $X$ and the weak topology of $\mathcal{R}(T,X)$ is the topology of weak convergence of probability measures in $\Pi(X)$. Denote by $\overline{\K}^{\,\mathit{w}}$ the weak closure of $\K\subset \mathcal{R}(T,X)$.

\subsection{The Purification Principle in Saturated Measure Spaces} 
A finite measure space $(T,\Sigma,\mu)$ is said to be \textit{essentially countably generated} if its $\sigma$-\hspace{0pt}algebra can be generated by a countable number of subsets together with the null sets; $(T,\Sigma,\mu)$ is said to be \textit{essentially uncountably generated} whenever it is not essentially countably generated. Let $\Sigma_S=\{ A\cap S\mid A\in \Sigma \}$ be the $\sigma$-\hspace{0pt}algebra restricted to $S\in \Sigma$. Denote by $L^1_S(\mu)$ the space of $\mu$-integrable functions on the measurable space $(S,\Sigma_S)$ whose element is identified with a restriction of a function in $L^1(\mu)$ to $S$. An equivalence relation $\sim$ on $\Sigma_S$ is given by $A\sim B \Leftrightarrow \mu(A\triangle B)=0$, where $A\triangle B$ is the symmetric difference of $A$ and $B$ in $\Sigma$. The collection of equivalence classes is denoted by $\Sigma(\mu)=\Sigma/\sim$ and its generic element $\widehat{A}$ is the equivalence class of $A\in \Sigma$. We define the metric $\rho$ on $\Sigma(\mu)$ by $\rho(\widehat{A},\widehat{B})=\mu(A\triangle B)$. Then $(\Sigma(\mu),\rho)$ is a complete metric space (see \citet[Lemma 13.13]{ab06} or \citet[Lemma III.7.1]{ds58}) and $(\Sigma(\mu),\rho)$ is separable if and only if $L^1(\mu)$ is separable (see \citet[Lemma 13.14]{ab06}). The \textit{density} of $(\Sigma(\mu),\rho)$ is the smallest cardinal number of the form $|\U|$, where $\U$ is a dense subset of $\Sigma(\mu)$. 

\begin{dfn}
A finite measure space $(T,\Sigma,\mu)$ is \textit{saturated} if $L^1_S(\mu)$ is nonseparable for every $S\in \Sigma$ with $\mu(S)>0$. We say that a finite measure space has the \textit{saturation property} if it is saturated.
\end{dfn}

Saturation implies nonatomicity and several equivalent definitions for saturation are known; see \cite{fk02,fr12,hk84,ks09}. One of the simple characterizations of the saturation property is as follows. A finite measure space $(T,\Sigma,\mu)$ is saturated if and only if $(S,\Sigma_S,\mu)$ is essentially uncountably generated for every $S\in \Sigma$ with $\mu(S)>0$. The saturation of finite measure spaces is also synonymous with the uncountability of the density of $\Sigma_S(\mu)$ for every $S\in \Sigma$ with $\mu(S)>0$; see \citet[331Y(e)]{fr12}. An germinal notion of saturation already appeared in \cite{ka44,ma42}. The significance of the saturation property lies in the fact that it is necessary and sufficient for the weak compactness and the convexity of the Bochner integral of a multifunction as well as the Lyapunov convexity theorem in Banach spaces; see \cite{ks13,ks15,ks16,po08,sy08}. 

Let $E$ be a Banach space and $L^1(\mu,E)$ be the space of Bochner integrable functions from $T$ to $E$. We say that a function $\Phi:T\times X\to E$ is \textit{integrably bounded} if there exists $\varphi\in L^1(\mu)$ such that $\| \Phi(t,x) \|\le \varphi(t)$ for every $(t,x)\in T\times X$. Hence, $\Phi(\cdot,x)\in L^1(\mu,E)$ for every $x\in X$ whenever $\Phi$ is integrably bounded and measurable, and $E$ is separable. Except for Subsections \ref{subsec3} and \ref{subsec4}, and Section \ref{appdx2}, the integration of $E$-valued functions with respect to the finite measure $\mu$ and probability measures in $\Pi(X)$ is always supposed to be in the Bochner sense. 

The following result is an immediate consequence of \citet[Theorem 5.1]{ks14b}, whose proof hinges on the Lyapunov convexity theorem in separable Banach spaces obtained in \cite{ks13} under the saturation hypothesis.

\begin{prop}[purification principle]
\label{PP1}
Let $(T,\Sigma,\mu)$ be a saturated finite measure space, $E$ be a separable Banach space, and $X$ be a compact Polish space. If $\Phi:T\times X\to E$ is an integrably bounded measurable function such that $\Phi(t,\cdot):X\to E$ is continuous in the weak topology of $E$ for every $t\in T$ and $U:T\twoheadrightarrow X$ is a multifunction with $\mathrm{gph}\,U\in \Sigma\otimes \mathrm{Borel}(X)$, then for every $\lambda\in \mathcal{R}(T,X)$ with $\lambda(t)(U(t))=1$ a.e.\ $t\in T$, there exists $f\in \M(T,X)$ with $f(t)\in U(t)$ a.e. $t\in T$ such that
$$
\int_T\int_X\Phi(t,x)\lambda(t,dx)d\mu=\int_T\Phi(t,f(t))d\mu.
$$
\end{prop}

A control-theoretic interpretation of Proposition \ref{PP1} means that any ``relaxed'' control system $t\mapsto\hat{\Phi}(t,\lambda(t)):=\int\Phi(t,x)\lambda(t,dx)$ operated by $\lambda\in \mathcal{R}(T,X)$ consistent with the control set $U(t)$ is realized by adopting a ``purified'' control system $t\mapsto\Phi(t,f(t))$ operated by $f\in \M(T,X)$ with the feasibility constraint $f(t)\in U(t)$ in such a way that its Bochner integral over $T$ is preserved with $\int\hat{\Phi}(t,\lambda(t))d\mu=\int\Phi(t,f(t))d\mu$. An application of Proposition \ref{PP1} to nonconvex variational problems with infinite-dimensional control systems is explored in \cite{ks14b}. 

\begin{rem}
\label{rem1}
For the case with $E=\R^n$, Proposition \ref{PP1} holds under the nonatomicity hypothesis, which is a well-known result in control theory attributed to \citet[Theorem IV.3.14]{wa72}; see also \citet[Theorem 2.5]{al72}. In particular, when $X$ is a finite or countably infinite set, Warga's result corresponds to the classical result of \citet{dww51a}; see also \cite{kr09,krs06}. The case for $E=\R^\N$ with $X$ a compact Polish space is covered in \citet[Theorem 2.2]{ls06}, \citet[Theorem 2.2]{ls09}, and \citet[Theorem 2]{po09} under the saturation hypothesis. As well as applications in optimal control theory along the lines of \cite{al72,ba84,bl73,ks14b,sa16,sb78,wa72}, the purification principle of this type also justifies the elimination of randomness in statistical decision theory as in   \cite{ba85,dww51b,fp06,gh05}, and the purification of mixed strategies for games with incomplete information with finite players, as in \cite{ast13,kr09,krs06,ls06,mw85,no14,rr82}. 
\end{rem}

\subsection{Gelfand Integrals in $L^\infty$}
\label{subsec3}
Let $(\Omega,\F,\nu)$ be a $\sigma$-finite measure space. A function $f:T\to L^\infty(\nu)$ is \textit{weakly$^*\!$ scalarly measurable} if the scalar function $\langle \varphi,f(\cdot) \rangle$ on $T$ is measurable for every $\varphi\in L^1(\nu)$, where the duality between $L^1(\nu)$ and $L^\infty(\nu)$ is given by $\langle \varphi,\psi \rangle=\int\varphi\psi d\nu$ for $\varphi\in L^1(\nu)$ and $\psi\in L^\infty(\nu)$. We say that weakly$^*\!$ scalarly measurable functions $f$ and $g$ are \textit{weakly$^*\!$ scalarly equivalent} if $\langle \varphi,f(t)-g(t) \rangle=0$ for every $\varphi\in L^1(\nu)$ a.e.\ $t\in T$ (the exceptional $\mu$-\hspace{0pt}null set depending on $\varphi$). We say that a weakly$^*\!$ scalarly measurable function $f:T\to L^\infty(\nu)$ is \textit{weakly$^*\!$ scalarly integrable} if the scalar function $\langle \varphi,f(\cdot) \rangle$ is integrable for every $\varphi\in L^1(\nu)$. A weakly$^*\!$ scalarly measurable function $f$ is \textit{Gelfand integrable} over $A\in \Sigma$ if there exists $\psi_A\in L^\infty(\nu)$ such that $\langle \varphi,\psi_A \rangle=\int_A\langle \varphi,f(t) \rangle d\mu$ for every $\varphi\in L^1(\nu)$. The element $\psi_A$ is called the \textit{Gelfand integral} (or the \textit{weak$^*\!$ integral}) of $f$ over $A$ and denoted by $\int_Afd\mu$. Every weakly$^*\!$ scalarly integrable function is weakly$^*\!$ integrable; see \citet[Theorem 11.52]{ab06}. Denote by $G^1(\mu,L^\infty(\nu))$ the equivalence classes of Gelfand integrable functions with respect to weakly$^*\!$ scalarly equivalence.  

Equipped with the notion of Gelfand integration, we turn to the development of a straightforward variant of Proposition \ref{PP1}. It is  a special case of \citet[Theorem 3.3]{sa16} to a setting where the integration of $L^\infty(\nu)$-valued functions with respect to the finite measure $\mu$ and probability measures in $\Pi(X)$ is always supposed to be in the sense of a Gelfand integral. 

\begin{prop}[purification principle in $L^\infty$]
\label{PP2}
Proposition \ref{PP1} is valid in the sense of Gelfand integrals when the separable Banach space E is replaced by $L^\infty(\nu)$ endowed with the weak$^*$-topology, where $(\Omega,\F,\nu)$ is a countably generated $\sigma$-finite measure space. 
\end{prop}

\section{Relaxed Large Economies}
\subsection{Relaxation of Large Economies}
The set of agents is given by a complete finite measure space $(T,\Sigma,\mu)$. The commodity space is given by a separable Banach space $E$. The preference relation ${\succsim}(t)$ of each agent $t\in T$ is a complete, transitive binary relation on a common consumption set $X\subset E$, which induces the preference map $t\mapsto {\succsim}(t)\subset X\times X$. We denote by $x\,{\succsim}(t)\,y$ the relation $(x,y)\in {\succsim}(t)$. The indifference and strict relations are defined respectively by $x\,{\sim}(t)\,y$ $\Leftrightarrow$ $x\,{\succsim}(t)\,y$ and $y\,{\succsim}(t)\,x$, and by $x\,{\succ}(t)\,y$ $\Leftrightarrow$ $x\,{\succsim}(t)\,y$ and $x\,{\not\sim}(t)\,y$. Each agent possesses an initial endowment $\omega(t)\in X$, which is the value of a Bochner integrable function $\omega:T\to E$. The economy $\E$ consists of the primitives $\E=\{ (T,\Sigma,\mu),X,\succsim,\omega \}$. 

The standing assumption on $\E$ is described as follows. 

\begin{assmp}
\label{assmp1}
\begin{enumerate}[(i)]
\item $X$ is a weakly compact subset of $E$. 
\item ${\succsim}(t)$ is a weakly closed subset of $X\times X$ for every $t\in T$. 
\item For every $x,y\in X$ the set $\{ t\in T\mid x\,{\succsim}(t)\,y \}$ is in $\Sigma$.  
\end{enumerate}
\end{assmp}

The weak compactness assumption in condition (i) is made in \cite{ky91,le13,mdr03,no97,po97,ry91} for the uncommon consumption set of each agent. Since $E$ is separable, the weakly compact set $X\subset E$ is metrizable for the weak topology (see \citet[Theorem V.6.3]{ds58}), and hence, the common consumption set $X$ is a compact Polish space. The preference relation ${\succsim}(t)$ is said to be \textit{continuous} if it satisfies condition (ii). The measurability of the preference mapping in condition (iii) is introduced in \cite{au69}. 

It follows from \citet[Proposition 1]{au69} that there exists a Carath\'{e}odory function $u:T\times X\to \R$ such that\footnote{While \cite{au69} treated the case where $X$ is the nonnegative orthant of a finite-dimensional Euclidean space, the proof is obviously valid as it stands for the case where $X$ is a separable metric space.}  
\begin{equation}
\label{rp1}
\forall x,y\in X\ \forall t\in T: x\,{\succsim}(t)\,y \Longleftrightarrow u(t,x)\ge u(t,y). 
\end{equation} 
Moreover, this representation in terms of Carath\'eodory functions is unique up to strictly increasing, continuous transformations in the following sense: If $F:T\times \R\to \R$ is a function such that $t\mapsto F(t,r)$ is measurable and $r\mapsto F(t,r)$ is strictly increasing and continuous, then $x\,{\succsim}(t)\,y \Leftrightarrow F(t,u(t,x))\ge F(t,u(t,y))$, where $(t,x)\mapsto F(t,u(t,x))$ is a Carath\'eodory function. In the sequel, we may assume without loss of generality that the preference map $t\mapsto {\succsim}(t)$ is represented by a Carath\'eodory function $u$ that is unique up to strictly increasing, continuous transformations. 

Following \cite{sa16}, we introduce the notion of ``relaxation'' of preferences for large economies. Given a continuous preference ${\succsim}(t)$ on $X$, its continuous affine extension ${\succsim}_\mathcal{R}(t)$ to $\Pi(X)$ is obtained by convexifying (randomizing) the individual utility function $u(t,\cdot)$ in such a way 
\begin{equation}
\label{rp2}
\forall P,Q\in \Pi(X)\ \forall t\in T: P\,{\succsim}_\mathcal{R}(t)\,Q \stackrel{\text{def}}{\Longleftrightarrow}  \int_X u(t,x)dP\ge \int_X u(t,x)dQ.
\end{equation}
The continuous extension ${\succsim}_{\mathcal{R}}(t)$ of ${\succsim}(t)$ from $X$ to the \textit{relaxed consumption set} $\Pi(X)$ is called a \textit{relaxed preference relation} on $\Pi(X)$. Thus, the restriction of ${\succsim_\mathcal{R}}(t)$ to $\Delta(X)$ coincides with ${\succsim}(t)$ on $X$. Indifference relation ${\sim}_\mathcal{R}(t)$ and strict relation ${\succ}_\mathcal{R}(t)$ are defined in a way analogous to the above. The extension formula \eqref{rp2} conforms to the relaxation technique explored in \cite{mc67,wa72,yo69}. It is noteworthy that relaxed preferences also conform to the ``expected utility hypothesis'' and the continuous function $u(t,\cdot)$ corresponds to the ``von Neumann--Morgenstern utility function'' for ${\succsim}_\mathcal{R}(t)$. That is, ${\succsim}_\mathcal{R}(t)$ is a continuous preference relation on $\Pi(X)$ satisfying the ``independence axiom'' introduced in \cite{vm53}. 
\begin{description}
\item[(Independence)] For every $P,Q,R\in \Pi(X)$ and $\alpha\in [0,1]$: $P\,{\sim}_\mathcal{R}(t)\,Q$ implies $\alpha P+(1-\alpha)R\,{\sim}_\mathcal{R}(t)\,\alpha Q+(1-\alpha)R$. 
\end{description}
Conversely, for every $t\in T$ any continuous binary relation on $\Pi(X)$ satisfying the independence axiom is representable in terms of the continuous von Neumann--Morgenstern utility function $u(t,\cdot)$ for which \eqref{rp2} is satisfied; see \citet[Theorem 3]{gr72}. Furthermore, this representation is unique up to positive affine transformations. 

Denote by $\E_\mathcal{R}=\{ (T,\Sigma,\mu),\Pi(X),{\succsim}_\mathcal{R},\delta_{\omega(\cdot)} \}$ the  \textit{relaxed economy} induced by the original economy $\E=\{ (T,\Sigma,\mu),X,\succsim,\omega \}$, where the initial endowment $\omega(t)\in X$ of each agent is identified with a Dirac measure $\delta_{\omega(t)}\in \Delta(X)$, and hence, $\delta_{\omega(\cdot)}\in \mathcal{R}(T,X)$.

\subsection{Relaxed Demand Sets}
Given a price $p\in E^*\setminus \{ 0 \}$, for each agent $t\in T$, as usual we define the budget set by $B(t,p)=\{ x\in X\mid \langle p,x \rangle\le \langle p,\omega(t) \rangle \}$ and the demand set by $D(t,p)=\{ x\in X\mid x\,{\succsim}(t)\,y \ \forall y\in B(t,p) \}$. Let $\imath_X$ be the identity map on $X$. Similarly, the \textit{relaxed budget set} of each agent is defined by
$$
B_\mathcal{R}(t,p)=\left\{ P\in \Pi(X) \mid \int_X \langle p,\imath_X(x) \rangle dP\le \langle p,\omega(t) \rangle \right\}
$$ 
and the \textit{relaxed demand set} is given by
$$
D_\mathcal{R}(t,p)=\{ P\in B_\mathcal{R}(t,p) \mid P\,{\succsim}_\mathcal{R}(t)\,Q \ \forall Q\in B_\mathcal{R}(t,p) \}. 
$$
We denote by $\int\imath_XdP$ the Bochner integral of $\imath_X$ with respect to the probability measure $P\in \Pi(X)$. Since $\int\langle p,\imath_X(x) \rangle dP=\langle p,\int\imath_XdP \rangle$ in view of the Bochner integrability of $\imath_X$, the ``barycentric commodity'' $\int\imath_XdP$ of $P\in B_\mathcal{R}(t,p)$ is in $X$ whenever $X$ is convex (which we do not assume), and affordable under the relaxed budget constraint, and the relaxed commodity $P$ is evaluated in terms of the expected utility represented in \eqref{rp2}. 

A remarkable, but natural connection between the market behavior of each agent in the original economy and that in the relaxed economy is that the maximization of expected utility subject to the relaxed budget constraint is ``consistent'' with the deterministic utility maximization subject to the budget constraint. Specifically, we have the following characterization on the relaxed demand set.  

\begin{prop}
\label{thm1}
Let $(T,\Sigma,\mu)$ be a finite measure space and $E$ be a separable Banach space. Suppose that the economy $\E$ satisfies Assumption \ref{assmp1}. Then for every $p\in E^*\setminus \{ 0 \}$ and $t\in T$: $P\in D_\mathcal{R}(t,p)$ if and only if $P(D(t,p))=1$. 
\end{prop}

\subsection{Relaxed Walrasian Equilibria}
To deal with the equilibrium concept with or without free disposal simultaneously, following \citet[Chapter 8]{mo06}, we introduce ``market constraints'' for the definition of (relaxed) allocations.  

\begin{dfn}
\label{Ball}
Let $C$ be a nonempty subset of $E$. 
\begin{enumerate}[(i)]
\item An element $f\in L^1(\mu,E)$ is an \textit{allocation} for $\E$ if it satisfies:
$$
\int_Tf(t)d\mu-\int_T\omega(t)d\mu\in C \quad\text{and $f(t)\in X$ a.e. $t\in T$}.
$$
\item An element $\lambda\in \mathcal{R}(T,X)$ is a \textit{relaxed allocation} for $\E_\mathcal{R}$ if it satisfies: 
$$
\int_T\int_X\imath_X(x)\lambda(t,dx)d\mu-\int_T\omega(t)d\mu\in C. 
$$   
\end{enumerate}
\end{dfn}

In particular, when $C=\{ 0 \}$, the definition reduces to the (relaxed) allocations ``without'' \textit{free disposal}; when $-C$ is a convex cone and $E$ is endowed with the cone order $\le$ defined by $x\le y \Leftrightarrow y-x\in -C$, the definition reduces to the (relaxed) allocations ``with'' free disposal. Denote by $\A(\E)$ the set of allocations for $\E$ and by $\A(\E_\mathcal{R})$ the set of relaxed allocations for $\E_\mathcal{R}$. If $\lambda$ is a relaxed allocation for $\E_\mathcal{R}$ such that $\lambda(t)=\delta_{f(t)}\in \Delta(X)$ for every $t\in T$ and $f\in L^1(\mu,E)$, then it reduces to the usual feasibility constraint $\int fd\mu-\int \omega d\mu\in C$ for $\E$. This means that $\A(\E)\subset \A(\E_\mathcal{R})$. 

\begin{dfn}
\begin{enumerate}[(i)]
\item A price-allocation pair $(p,f)\in (E^*\setminus \{ 0 \})\times \A(\E)$ is a \textit{Walrasian equilibrium} for $\E$ if a.e.\ $t\in T$: $f(t)\in B(t,p)$ and $f(t)\,{\succsim}(t)\,x$ for every $x\in B(t,p)$.
\item A price-relaxed allocation pair $(p,\lambda)\in (E^*\setminus \{ 0 \})\times \A(\E_\mathcal{R})$ is a \textit{relaxed Walrasian equilibrium} for $\E_\mathcal{R}$ if a.e.\ $t\in T$: $\lambda(t)\in B_\mathcal{R}(t,p)$ and $\lambda(t)\,{\succsim}_\mathcal{R}(t)\,P$ for every $P\in B_\mathcal{R}(t,p)$.
\end{enumerate}
\end{dfn}
\noindent
Denote by $\W(\E)$ the set of Walrasian allocations for $\E$ and by $\W(\E_\mathcal{R})$ the set of relaxed Walrasian allocations for $\E_\mathcal{R}$. 

Any Walrasian equilibrium for the original economy is regarded as a ``purified'' relaxed Walrasian equilibrium for the relaxed economy. Under the saturation hypothesis, the converse result holds as well. That is, any relaxed Walrasian equilibrium for the relaxed economy can be purified as a Walrasian equilibrium for the original economy.  

\begin{prop}
\label{eqv1}
Let $(T,\Sigma,\mu)$ be a finite measure space and $E$ be a separable Banach space. Suppose that the economy $\E$ satisfies Assumption \ref{assmp1}. If $(p,f)$ is a Walrasian equilibrium for $\E$, then $(p,\delta_{f(\cdot)})$ is a relaxed Walrasian equilibrium for $\E_\mathcal{R}$. Conversely, if $(p,\lambda)$ is a relaxed Walrasian equilibrium for $\E_\mathcal{R}$, then there exists a Walrasian equilibrium $(p,f)$ for $\E$ such that $\lambda(t)\,{\sim}_\mathcal{R}(t)\,\delta_{f(t)}$ a.e.\ $t\in T$ whenever $(T,\Sigma,\mu)$ is saturated. 
\end{prop}

Another significant aspect on saturation is the density property of allocations and Walrasian allocations.

\begin{prop}[density property]
\label{dens1}
Let $(T,\Sigma,\mu)$ be a saturated finite measure space and $E$ be a separable Banach space. Suppose that the economy $\E$ satisfies Assumption \ref{assmp1}. Then $\A(\E_\mathcal{R})=\overline{\A(\E)}^{\mathit{\,w}}$ and $\W(\E_\mathcal{R})=\overline{\W(\E)}^{\mathit{\,w}}$.
\end{prop}

\begin{rem}
\label{rem}
It is \citet[Theorem IV.2.6]{wa72} who established the density theorem $\mathcal{R}(T,X)=\overline{\M(T,X)}^{\,\mathit{w}}$ for compact polish spaces under the nonatomicity hypothesis. As noted in \citet[Remark 6.1]{ks14b}, Proposition \ref{dens1} holds under the nonatomicity hypothesis whenever $E=\R^n$, in which case the classical Lyapunov convexity theorem is sufficient for the density property. For another variant of the density property with the finite-dimensional setting, see, e.g., \citet[Corollary 3]{ba84}, \citet[Proposition II.7]{bl73}, and \citet[Theorem 7 and Corollary 4]{sb78}.  
\end{rem}

\subsection{Existence of Walrasian Equilibria with Free Disposal}
\label{subsec1}
For a substantive validation of the equivalence in Proposition \ref{eqv1}, it suffices to demonstrate the existence of relaxed Walrasian equilibria for the relaxed economy $\E_\mathcal{R}$ instead of Walrasian equilibria for the original economy $\E$. Following \cite{ky91,le13,mdr03,no97,po97,ry91}, we consider (relaxed) Walrasian equilibria with free disposal in which the commodity space $E$ is an ordered separable Banach space such that the norm interior of the positive cone $E_+$ is nonempty. Denote by $E_+^*$ be the set of elements $x^*\in E^*$ with $\langle x^*,x \rangle\ge 0$ for every $x\in E_+$. An  element in $E_+^*\setminus \{ 0 \}$ is said to be \textit{positive}. A maximal element in $X$ for ${\succsim}(t)$ is called a \textit{satiation point} for ${\succsim}(t)$. Under Assumption \ref{assmp1}, satiation points for ${\succsim}(t)$ exist for every $t\in T$. 

\begin{assmp}
\label{assmp2}
\begin{enumerate}[(i)]
\item $X$ is a weakly compact subset of $E_+$. 
\item For every $t\in T$ there exists $z(t)\in X$ such that $\omega(t)-z(t)$ belongs to the norm interior of $E_+$. 
\item If $x\in X$ is a satiation point for ${\succsim}(t)$, then $x\ge \omega(t)$.  
\item If $x\in X$ is not a satiation point for ${\succsim}(t)$, then $x$ belongs to the weak closure of the upper contour set $\{ y\in X\mid y\,{\succ}(t)\,x \}$. 
\end{enumerate}
\end{assmp}
\noindent
Condition (ii) is due to \cite{ky91}, which guarantees that for every positive price the value of the initial endowment of each agent is strictly positive. Condition (iii) is introduced in \cite{po97} and imposed also in \cite{le13}. Condition (iv) is a variant of ``local nonsatiation'' originated in \cite{hi68} and is imposed also in \cite{le13,po97}.

We now present the first substantive result of this paper. 

\begin{thm}
\label{RWE1}
Let $(T,\Sigma,\mu)$ be a finite measure space and $E$ be an ordered separable Banach space such that the norm interior of $E_+$ is nonempty. Then for every economy $\E$ satisfying Assumptions \ref{assmp1} and \ref{assmp2}:
\begin{enumerate}[\rm (i)]
\item There exists a relaxed Walrasian equilibrium with free disposal for $\E_\mathcal{R}$ with a positive price. 
\item There exists a Walrasian equilibrium with free disposal for $\E$ with a positive price whenever $(T,\Sigma,\mu)$ is saturated. 
\end{enumerate}
\end{thm}

A sharp contrast to the literature on large economies, such as \cite{au66,hi74,ky91,le13,mdr03,no97,po97,ry91}, is that the saturation (or even the nonatomicity) hypothesis is unnecessary to guarantee the existence of relaxed Walrasian equilibria for the relaxed economies as well as the convexity hypothesis. Thus, whenever the set $T$ of agents is finite and $\mu$ is a counting measure, the first assertion of Theorem \ref{RWE1} reduces to the existence of relaxed Walrasian equilibria for a relaxed finite economies without convexity assumptions. This means that relaxed Walrasian equilibria always exist even though the original economy fails to possess Walrasian equilibria. Since the vernacular of \lq\lq relaxed" economies and \lq\lq relaxed" Walrasian equilibria is also used in \cite{ba08}, we invite the reader to compare Theorem \ref{RWE1} with the relevant result in \cite{ba08}.\footnote{\label{fn1}Indeed,  \citet{ba08} has the priority for the usage of this terminology in mathematical economics, but as mentioned in Footnote \ref{fn0}, rather than classical Walrasian general equilibrium theory, his concern is with a synthetic treatment that allows externalities and price-dependent preferences with a finite-dimensional commodity space.}

Given Proposition \ref{eqv1}, the second assertion of Theorem \ref{RWE1} simply drops the convexity hypothesis from \cite{ky91} under the saturation hypothesis and recovers the existence result of \cite{le13,po97,ry91} under the framework of economies with a common consumption set. Indeed, when $E$ is a finite-dimensional Euclidean space, the validity of the second assertion of Theorem \ref{RWE1} for nonatomic finite measure space of agents follows from Proposition \ref{PP1}; see Remark \ref{rem1}. Therefore, to repeat, if $(T,\Sigma,\mu)$ is a nonatomic finite measure space and $E=\R^n$, the existence of Walrasian equilibria with free disposal for $\E$ with a positive price is guaranteed under Assumptions \ref{assmp1} and \ref{assmp2}. On the other hand, when $X$ is a finite subset of $E_+$, the conditions on   (non)satiation points for ${\succsim}(t)$ are unnecessary for the existence result and Assumption \ref{assmp2} can be replaced by the following.  

\begin{assmp}
\label{assmp3}
\begin{enumerate}[\rm (i)]
\item $X$ is a finite subset of $E_+$.
\item For every $t\in T$ there exists $z(t)\in X$ such that $\omega(t)-z(t)$ belongs to the norm interior of $E_+$. 
\end{enumerate} 
\end{assmp}
\noindent 
This allows us to present the second substantive result of the paper.  

\begin{thm}
\label{RWE2}
Let $(T,\Sigma,\mu)$ be a finite measure space and $E$ be an ordered separable Banach space such that the norm interior of $E_+$ is nonempty. Then for every economy $\E$ satisfying Assumptions \ref{assmp1} and \ref{assmp3}: 
\begin{enumerate}[\rm (i)]
\item There exists a relaxed Walrasian equilibrium with free disposal for $\E_\mathcal{R}$ with a positive price.
\item There exists a Walrasian equilibrium with free disposal for $\E$ with a positive price whenever $(T,\Sigma,\mu)$ is saturated.
\end{enumerate}
\end{thm}

We now conclude this subsection with two applications.

\begin{exmp}[envy-freeness/incentive compatibility]
An allocation $f\in \A(\E)$ is said to be \textit{envy-free} if $f(t)\,{\succsim}(t)\,f(t')$ for a.e.\ $t,t'\in T$. Let $\bar{\omega}(t)=\int\omega d\mu/\mu(T)$ and assume that $\bar{\omega}(t)\in X$ for every $t\in T$. If Assumptions \ref{assmp1} and \ref{assmp2} are satisfied for the economy $\overline{\E}=\{ (T,\Sigma,\mu),{\succsim},X,\bar{\omega} \}$ with the same initial endowment among agents, then Theorem \ref{RWE1} guarantees that $\overline{\E}$ possesses a Walrasian equilibrium that is also Pareto optimal and envy-free; see \cite{va74}. When $T$ is regarded as the set of random shocks drawn from the probability measure $\mu$, where each element $t\in T$ is an idiosyncratic shock that characterizes the type of agent, the envy-free condition is reduced to the ``truth revelation principle'', i.e., the ``incentive compatibility'' condition studied in \cite{pt84}. This reduces to, and implies,  the existence of Walrasian equilibria with incentive compatibility for economies with private information under the saturation hypothesis.  
\end{exmp}

\begin{exmp}[indivisible commodities]
Suppose that there are $n$ indivisible commodities each of which can be consumed in integer units and that the common consumption set of such commodities is finite for all agents. The resulting economy $\E$ with indivisible commodities is described in our framework as follows. Let $\Z_+$ be the set of nonnegative integers, $E$ the Euclidean space $\R^n$ with the Euclid norm, and  $X$ a finite subset of $\Z_+^n$. Let ${\succsim}(t)$ be a preference on $X$ represented by a Carath\'eodory function $u:T\times X\to \R$. Assume further that the endowment function $\omega:T\to \R^n$ is integrable such that $\omega(t)$ belongs to $X$ and each of its coordinates is a positive integer for every $t\in T$. Then Assumptions \ref{assmp1} and \ref{assmp3} are automatically satisfied, and Theorem \ref{RWE2} then guarantees that there exists a relaxed Walrasian equilibrium with free disposal for the relaxed economy $\E_\mathcal{R}$ with a positive price. In particular, if $(T,\Sigma,\mu)$ is nonatomic, then there exists a Walrasian equilibrium with free disposal for $\E$ with a positive price. The crucial difference of this consequence from the existence result in  \cite{ky81} is that it dispenses with the introduction of divisible commodities and  with the local nonsatiation of preferences, though at the cost of the finiteness of the consumption set. 
\end{exmp}

\subsection{Existence of Walrasian Equilibria on $L^\infty$}
\label{subsec4}
In this subsection, we turn to economies that are modeled with $L^\infty$ as a commodity. This extension is important for both substantive and technical reasons: substantively because it, and its dual,  was identified by \citet{be72} as the canonical space for Walrasian general equilibrium theory; technically because it leads to to shift the emphasis of analysis to the predual $L^1$ rather than the dual, and thereby from Bochner integration to Gelfand integration.  

Towards this end, let $(\Omega,\F,\nu)$ be a countably generated, $\sigma$-finite measure space. The norm dual of $L^\infty(\nu)$ is $\mathit{ba}(\nu)$, the space of finitely additive signed measures on $\F$ of bounded variation that vanishes on $\nu$-null sets with the duality given by $\langle \pi,\psi \rangle=\int\psi d\pi$ for $\pi\in \mathit{ba}(\nu)$ and $\psi\in L^\infty(\nu)$; see \citet[Theorem IV.8.14]{ds58}. We consider a (Gelfand) economy $\E^G=\{ (T,\Sigma,\mu),X,{\succsim},\omega \}$ for which the commodity space is $L^\infty(\nu)$ and the price space is $\mathit{ba}(\nu)$ with $\omega\in G^1(\mu,L^\infty(\nu))$ and $\omega(t)\in X$ for every $t\in T$ satisfying the following conditions. 

\begin{assmp}
\label{assmp4}
\begin{enumerate}[(i)]
\item $X$ is a weakly$^*\!$ compact subset of $L^\infty(\nu)$. 
\item ${\succsim}(t)$ is a weakly$^*$ closed subset of $X\times X$ for every $t\in T$. 
\item For every $x,y\in X$ the set $\{ t\in T\mid x\,{\succsim}(t)\,y \}$ is in $\Sigma$. 
\end{enumerate}
\end{assmp}

\noindent Since $L^1(\nu)$ is separable, the weak$^*$ compact set $X\subset L^\infty_+(\nu)$ is metrizable for the weak$^*$ topology of $L^\infty(\nu)$ (see \citet[Theorem V.5.1]{ds58}), and hence, the common consumption set $X$ is a compact Polish space. Therefore, the preference representation in \eqref{rp1} is valid for $\E^G$. Consequently, the preference representation \eqref{rp2} is also valid for its relaxed economy $\E^G_\mathcal{R}=\{ (T,\Sigma,\mu),\Pi(X),{\succsim}_\mathcal{R},\omega \}$. 

Next, we develop the analogue for Definition \ref{Ball}. 

\begin{dfn} 
Let $C$ be a nonempty subset of $L^\infty(\nu)$. 
\begin{enumerate}[(i)]
\item An element $f\in G^1(\mu,L^\infty(\nu))$ is an \textit{allocation} for $\E^G$ if it satisfies:
$$
\int_Tf(t)d\mu-\int_T\omega(t)d\mu\in C \quad\text{and $f(t)\in X$ a.e. $t\in T$}.
$$
\item An element $\lambda\in \mathcal{R}(T,X)$ is a \textit{relaxed allocation} for $\E^G_\mathcal{R}$ if it satisfies: 
$$
\int_T\int_X\imath_X(x)\lambda(t,dx)d\mu-\int_T\omega(t)d\mu\in C. 
$$   
\end{enumerate}
\end{dfn}
\noindent
Denote by $\A(\E^G)$ the set of Gelfand integrable allocations for $\E^G$ and by $\A(\E^G_\mathcal{R})$ the set of relaxed allocations for $\E^G_\mathcal{R}$. 

Given a price $\pi\in \mathit{ba}(\nu)\setminus \{ 0 \}$, we can define (relaxed) budget set and (relaxed) demand set for each agent as in the previous section. Thus, (relaxed) Walrasian equilibria with free disposal for $\E^G$ (resp.\ $\E^G_\mathcal{R}$) are introduced in an obvious way. 

\begin{dfn}
\begin{enumerate}[(i)]
\item A price-allocation pair $(\pi,f)\in (\mathit{ba}(\nu)\setminus \{ 0 \})\times \A(\E^G)$ is a \textit{Walrasian equilibrium} for $\E^G$ if a.e.\ $t\in T$: $f(t)\in B(t,\pi)$ and $f(t)\,{\succsim}(t)\,x$ for every $x\in B(t,\pi)$.
\item A price-relaxed allocation pair $(\pi,\lambda)\in (\mathit{ba}(\nu)\setminus \{ 0 \})\times \A(\E^G_\mathcal{R})$ is a \textit{relaxed Walrasian equilibrium} for $\E^G_\mathcal{R}$ if a.e.\ $t\in T$: $\lambda(t)\in B_\mathcal{R}(t,\pi)$ and $\lambda(t)\,{\succsim}_\mathcal{R}(t)\,P$ for every $P\in B_\mathcal{R}(t,\pi)$.
\end{enumerate}
\end{dfn}
\noindent
Denote by $\W(\E^G)$ the set of Walrasian allocations for $\E^G$ and by $\W(\E^G_\mathcal{R})$ the set of relaxed Walrasian allocations for $\E^G_\mathcal{R}$. 

It is clear now that Proposition \ref{thm1} is valid for $E=L^\infty(\nu)$ and $X\subset L^\infty(\nu)$ under Assumption \ref{assmp4}. Corresponding to Proposition \ref{eqv1}, we obtain the following characterization under the saturation hypothesis, whose proof is same with that of Proposition \ref{eqv1} if one simply replaces the Bochner integrals by Gelfand ones.

\begin{prop}
\label{eqv2}
Let $(T,\Sigma,\mu)$ be a saturated finite measure space and $(\Omega,\F,\nu)$ be a countably generated $\sigma$-finite measure space. Suppose that the economy $\E^G$ satisfies Assumption \ref{assmp4}. If $(\pi,f)$ is a Walrasian equilibrium for $\E^G$, then $(\pi,\delta_{f(\cdot)})$ is a relaxed Walrasian equilibrium for $\E^G_\mathcal{R}$. Conversely, if $(\pi,\lambda)$ is a relaxed Walrasian equilibrium for $\E^G_\mathcal{R}$, then there exists a Walrasian equilibrium $(\pi,f)$ for $\E^G$ such that $\lambda(t)\,{\sim}_\mathcal{R}(t)\,\delta_{f(t)}$ a.e.\ $t\in T$ whenever $(T,\Sigma,\mu)$ is saturated. 
\end{prop}

Furthermore, under the same hypothesis with Proposition \ref{eqv2}, the density property in $L^\infty(\nu)$ is also valid for the Gelfand integral setting, i.e., $\A(\E^G_\mathcal{R})=\overline{\A(\E^G)}^{\mathit{\,w}}$ and $\W(\E^G_\mathcal{R})=\overline{\W(\E^G)}^{\mathit{\,w}}$. The proof of this fact is same with that of Proposition \ref{dens1} if one simply replaces the Bochner integrals by Gelfand ones invoking Proposition \ref{PP2}.

Next, we turn to the analogues of our substantive Theorems \ref{RWE1} on the  existence of (relaxed) Walrasian equilibria in large economies with free disposal modeled on $L^\infty(\nu)$ as a commodity space. Since the norm interior of the positive cone $L^\infty_+(\nu)$ of $L^\infty(\nu)$ is nonempty, under the additional assumption below, we can recover every result in Subsection \ref{subsec1} for the case with $E=L^\infty(\nu)$ and $E^*=\mathit{ba}(\nu)$ with the suitable replacement of the weak topology by the weak$^*$ topology and the Bochner integrals by the Gelfand integrals. 

\begin{assmp}
\label{assmp5}
\begin{enumerate}[(i)]
\item $X$ is a weakly$^*\!$ compact subset of $L^\infty_+(\nu)$. 
\item For every $t\in T$ there exists $z(t)\in X$ such that $\omega(t)-z(t)$ belongs to the norm interior of $L^\infty_+(\nu)$. 
\item If $x\in X$ is a satiation point for ${\succsim}(t)$, then $x\ge \omega(t)$.  
\item If $x\in X$ is not a satiation point for ${\succsim}(t)$, then $x$ belongs to the weak$^*$ closure of the upper contour set $\{ y\in X\mid y\,{\succ}(t)\,x \}$. 
\end{enumerate}
\end{assmp}

While the norm dual $\mathit{ba}(\nu)$ of $L^\infty(\nu)$ is larger than $L^1(\nu)$, as emphasized in \citet{be72}, the price systems in $\mathit{ba}(\nu)$ lack a reasonable economic interpretation unless they belong to $L^1(\nu)$ (i.e., they are countably additive); see also \cite{mz91}. To derive positive equilibrium prices with free disposal in $L^1(\nu)$ for the relaxed economy from those in $\mathit{ba}(\nu)$, the Yosida--Hewitt decomposition of finitely additive measures is crucial in our framework, similar to \cite{be72,be91}. 

\begin{thm}
\label{WE3}
Let $(T,\Sigma,\mu)$ be a finite measure space and $(\Omega,\F,\nu)$ be a countably generated $\sigma$-finite measure space. Then for every economy $\E^G$ satisfying Assumptions \ref{assmp4} and \ref{assmp5}:
\begin{enumerate}[\rm (i)]
\item There exists a relaxed Walrasian equilibrium with free disposal for $\E^G_\mathcal{R}$ with a positive price in $L^1(\nu)$.
\item There exists a Walrasian equilibrium with free disposal for $\E^G$ with a positive price in $L^1(\nu)$ whenever $(T,\Sigma,\mu)$ is saturated. 
\end{enumerate}
\end{thm}

The second assertion of Theorem \ref{WE3} removes the convexity and monotonicity of preferences from \cite{be72,be91} and introduces free disposability for Walrasian equilibria with the commodity space of $L^\infty(\nu)$.

\section{Concluding Summary} 
In this paper we have presented three results as our contribution to the existence question of classical Walrasian general equilibrium with a continuum of commodities and a finite or a continuum of agents. We have developed these results through four propositions emphasizing  \lq\lq convexification by randomization'' as opposed to \lq\lq convexification by aggregation'', and drawn on relaxation techniques pioneered, and now pervasive, in optimal control theory as well as in statistical decision-theory, and the theory of non-cooperative games.   To be sure, the substantive power of these techniques lies in their being supplemented by a purification principle that eliminates randomization in the testing of statistical hypothesis and in the replacement of mixed strategies by their \lq\lq equivalent'' pure strategies in both non-atomic and atomic game theory. All this being said, our thrust is squarely on classical Walrasian general equilibrium theory as formulated by Aumann in 1964--1966.
It is this focus that leads us to ignore a possible  third convexification procedure due to Hart--Hildenbrand--Kohlberg; see \cite{hi74} and its references. This approach   involves a substantive  shift from an anonymous to a non-anonymous form for both games and economies, and a technical shift from measurable functions to their induced distributions.  It substitutes a \lq\lq symmetrization principle" for the purification principle. We leave a consideration of this developing and rich literature for future work.

\appendix
\section{Appendix 1}
\subsection{Proof of Proposition \ref{thm1}}
Choose any $P\in D_\mathcal{R}(t,p)$. Given the preference representation \eqref{rp2}, if $P(D(t,p))<1$, then 
\begin{equation}
\label{eq1}
\int_Xu(t,x)dP=\int_{D(t,p)}u(t,x)dP+\int_{X\setminus D(t,p)}u(t,x)dP<\max_{y\in B(t,p)}u(t,y)
\end{equation}
because $u(t,x)=\max_{y\in B(t,p)}u(t,y)$ for every $x\in D(t,p)$ and $u(t,x)<\max_{y\in B(t,p)}u(t,y)$ for every $x\in X\setminus D(t,p)$. On the other hand, for every $y\in B(t,p)$ we have
$$
\int_Xu(t,x)dP=\max_{Q\in B_\mathcal{R}(t,p)}\int_Xu(t,x)dQ\ge \int_Xu(t,x)d\delta_y=u(t,y)
$$
in view of $\delta_y\in B_\mathcal{R}(t,p)$. Hence, we obtain a contradiction because of $\int u(t,x)dP\ge \max_{y\in B(t,p)}u(t,y)$. 

For the converse implication, suppose that $P(D(t,p))=1$. Since $\langle p,x\rangle\le \langle p,\omega(t) \rangle$ for every $x\in D(t,p)$, we have 
$$
\int_X\langle p,\imath_X(x) \rangle dP=\int_{D(t,p)}\langle p,\imath_X(x) \rangle dP\le \int_{D(t,p)}\langle p,\omega(t) \rangle dP=\langle p,\omega(t) \rangle. 
$$
Thus, if $P$ does not belong to $D_\mathcal{R}(t,p)$, then there exists $Q\in B_\mathcal{R}(t,p)$ such that $\int u(t,x)dQ>\int u(t,x)dP$. Note also that  
$$
\int_Xu(t,x)dP=\int_{D(t,p)}u(t,x)dP=\int_{D(t,p)}\max_{y\in B(t,p)}u(t,y)dP=\max_{y\in B(t,p)}u(t,y). 
$$
Furthermore, $Q(D(t,p))=1$ a.e.\ $t\in T$; for otherwise, we have $\int u(t,x)dQ<\max_{y\in B(t,p)}u(t,y)$ as derived in \eqref{eq1}, a contradiction. We thus obtain
\begin{align*}
\int_Xu(t,x)dQ=\int_{D(t,p)}u(t,x)dQ
& =\int_{D(t,p)}\max_{y\in B(t,p)}u(t,y)dQ=\max_{y\in B(t,p)}u(t,y). 
\end{align*}
This is a contradiction to the initial hypothesis. Therefore, $P\in D_\mathcal{R}(t,p)$. \qed

\subsection{Proof of Proposition \ref{eqv1}}
Pick any Walrasian equilibrium $(p,f)$ for $\E$. If the price-relaxed allocation pair $(p,\delta_{f(\cdot)})$ is not a relaxed Walrasian equilibrium for $\E_\mathcal{R}$, then there exists $A\in \Sigma$ of positive measure such that for every $t\in A$ there exists $P\in B_\mathcal{R}(t,p)$ with $P\,{\succ}_\mathcal{R}(t)\,\delta_{f(t)}$. Given the preference representation \eqref{rp2}, this means the inequality $\int u(t,x)dP>u(t,f(t))=\max_{y\in B(t,p)}u(t,y)$. We then have $P(D(t,p))=1$ for every $t\in A$; for otherwise, $\int u(t,x)dP<\max_{y\in B(t,p)}u(t,y)$ for some $t\in A$, a contradiction. On the other hand, the equalities
$$
\int_Xu(t,x)dP=\int_{D(t,p)}u(t,x)dP=\max_{y\in B(t,p)}u(t,y)
$$
for every $t\in T$ yield a contradiction to the above inequality. Therefore, $(p,\delta_{f(\cdot)})$ is a relaxed Walrasian equilibrium for $\E_\mathcal{R}$. 

Take any relaxed Walrasian equilibrium $(p,\lambda)$ for $\E_\mathcal{R}$. Let $\mathrm{gph}\,B(\cdot,p)$ (resp.\ $\mathrm{gph}\,D(\cdot,p))$ be the graph of the multifunction $B(\cdot,p):T\twoheadrightarrow X$ (resp.\ $D(\cdot,p):T\twoheadrightarrow X$) and denote by $\mathrm{Borel}(E,\mathit{w})$ the Borel $\sigma$-algebra generated by the weak topology of $E$. Since 
$$
D(t,p)=\left\{ x\in X\mid u(t,x)=\max_{y\in B(t,p)}u(t,y) \right\}
$$
with $\mathrm{gph}\,B(\cdot,p)\in \Sigma\otimes \mathrm{Borel}(E,\mathit{w})$, the measurable maximum theorem (see \citet[Proposition 3, p.\,60]{hi74}) guarantees that $\mathrm{gph}\,D(\cdot,p)\in \Sigma\otimes \mathrm{Borel}(E,\mathit{w})$. By Proposition \ref{thm1}, we have $\lambda(t)(D(t,p))=1$ a.e.\ $t\in T$. It follows from Proposition \ref{PP1} that there exists $f\in \mathcal{M}(T,X)\subset L^1(\mu,E)$ with $f(t)\in D(t,p)$ a.e.\ $t\in T$ such that $\int fd\mu=\iint\imath_X\lambda(t,dx)d\mu$. Therefore, $(p,f)$ is a Walrasian equilibrium for $\E$. Since $\lambda(t)(D(t,p))=1$ a.e.\ $t\in T$ by Theorem \ref{thm1}, we have $\int u(t,x)\lambda(t,dx)=\max_{y\in B(t,p)}u(t,y)=u(t,f(t))$ a.e.\ $t\in T$. Therefore, $\lambda(t)\,{\sim}_\mathcal{R}(t)\,\delta_{f(t)}$ a.e.\ $t\in T$. \qed

\subsection{Proof of Proposition \ref{dens1}}
Let $\lambda_0\in \A(\E_\mathcal{R})$ be arbitrarily and $\mathcal{N}_0$ be its any  neighborhood. By definition of the weak topology, there exists $u_1,\dots,u_k$ in $\C^1(T\times X,\mu)$ such that $|I_{u_i}(\nu)-I_{u_i}(\nu_0)|<1$, $i=1,\dots,k$ implies $\nu\in \mathcal{N}_0$. Define $\Phi:T\times X\to E\times \R^k$ in Proposition \ref{PP1} by $\Phi=(\imath_X,u_1,\dots,u_k)$. Then there exists $f\in \M(T,X)$ such that $\iint\Phi(t,x)\lambda(t,dx)d\mu=\int \Phi(t,f(t))d\mu$. This means that $f\in L^1(\mu,E)$, $\iint\imath_X(x)\lambda(t,dx)d\mu=\int fd\mu$, and $I_{u_i}(\nu_0)=I_{u_i}(\delta_{f})$ for $i=1,\dots,k$. Therefore, $f\in \A(\E)$ and $\delta_f\in \mathcal{N}_0$. Since the choice of $\lambda_0$ and $\mathcal{N}_0$ is arbitrary, $\A(\E)$ is dense in $\A(\E_\mathcal{R})$. Next, let $\lambda_0\in \W(\E_\mathcal{R})$ be arbitrarily. Since $\lambda_0\in D_\mathcal{R}(t,p)$ a.e.\ $t\in T$ for some $p\in E^*\setminus \{ 0 \}$, it follows from Proposition \ref{thm1} that $\lambda_0(t)(D(t,p))=1$ a.e.\ $t\in T$. Let $U(t)\equiv D(t,p)$ in Proposition \ref{PP1} and $\mathcal{N}_0$ be any neighborhood of $\lambda_0$. Then as in the above there exists $f\in \A(\E)$ with $f(t)\in D(t,p)$ a.e.\ $t\in T$ such that $I_{u_i}(\nu_0)=I_{u_i}(\delta_{f})$ for $i=1,\dots,k$. Therefore, $f\in \W(\E)$ and $\delta_f\in \mathcal{N}_0$, and hence, $\W(\E)$ is dense in $\W(\E_\mathcal{R})$. \qed

\subsection{Proof of Theorem \ref{RWE1}}
\label{subsec2}
The set of normalized price functionals is given by $S^*=\{ p\in E^*_+\mid \langle p,v \rangle=1 \}$, where $v\in E_+$ is taken from the norm interior of $E_+$. Then the  Banach--Alaoglu theorem guarantees that $S^*$ is weakly$^*\!$ compact; see \citet[p.\,1859]{mz91}. 

\begin{lem}
\label{lem1}
$D_\mathcal{R}:T\times S^*\twoheadrightarrow \Pi(X)$ is a compact, convex-valued multifunction with $\mathrm{gph}\,D_\mathcal{R}(\cdot,p)\in \Sigma\otimes\mathrm{Borel}(\Pi(X))$ for every $p\in S^*$. 
\end{lem}

\begin{proof}
The compactness and convexity of $D_\mathcal{R}(t,p)$ follows from the continuity and affinity of the relaxed utility function $P\mapsto \int u(t,x)dP$ in \eqref{rp2}. Fix $p\in S^*$ arbitrarily and define $\theta_p:T\times \Pi(X)\to \R$ by $\theta_p(t,P)=\int\langle p,\imath_X(x) \rangle dP-\langle p,\omega(t) \rangle$. Then $t\mapsto \theta_p(t,P)$ is measurable for every $P\in \Pi(X)$. Since $x\mapsto \langle p,\imath_X(x) \rangle$ is a bounded continuous function on $X$, the function $P\mapsto \theta_p(t,P)$ is continuous for every $t\in T$ in view of the definition of the topology of weak convergence of probability measures. Thus, $\theta_p$ is a Carath\'eodory function, and hence, it is jointly measurable. Denote by $\mathrm{Borel}(\Pi(X))$ the Borel $\sigma$-algebra of $\Pi(X)$. We then have   
\begin{align*}
\mathrm{gph}\,B_\mathcal{R}(\cdot,p)
& =\left\{ (t,P)\in T\times \Pi(X) \mid \theta_p(t,P)\le 0\right\}\in \Sigma\otimes\mathrm{Borel}(\Pi(X)).
\end{align*} 
For the sake notational simplicity, define $\tilde{u}:T\times \Pi(X)\to \R$ by $\tilde{u}(t,P)=\int u(t,x)dP$. Then $\tilde{u}$ is a Carath\'eodory function, and hence, it is jointly measurable. Given the representation of the relaxed preferences in \eqref{rp2}, we have 
\begin{align*}
D_\mathcal{R}(t,p)=\left\{ P\in B_\mathcal{R}(t,p) \mid \tilde{u}(t,P)=\max_{Q\in B_\mathcal{R}(t,p)}\tilde{u}(t,Q) \right\}.
\end{align*}
Hence, by the measurable maximum theorem (see \citet[Proposition 3, p.\,60]{hi74}), we have $\mathrm{gph}\,D_\mathcal{R}(\cdot,p)\in \Sigma\otimes\mathrm{Borel}(\Pi(X))$. 
\end{proof}

A difficulty might arise in the derivation of the upper semicontinuity of $p\mapsto D_\mathcal{R}(t,p)$ because of the failure of the joint continuity of the valuation functional $(p,P)\mapsto \int \langle p,\imath_X(x) \rangle dP$ on $S^*\times \Pi(X)$ whenever $S^*$ is endowed with the weak$^*\!$ topology of $E^*$, which is analogous to the well-known failure of the joint continuity of the valuational functional $(p,x)\mapsto \langle p,x \rangle$ on $S^*\times E$ whenever $E$ is endowed with the weak topology; e.g., see \cite{mz91}. To overcome the difficulty of the upper semicontinuity of $D_\mathcal{R}(t,\cdot):S^*\twoheadrightarrow \Pi(X)$, we ``enlarge'' the relaxed demand set by introducing the multifunction $\Gamma:T\times S^*\twoheadrightarrow \Pi(X)$ defined by 
$$
\Gamma(t,p)=\{ P\in \Pi(X) \mid P\,{\succsim}_\mathcal{R}(t)\,Q \ \forall Q\in B_\mathcal{R}(t,p) \} 
$$
adapting the device used in \cite{le13,po97} to the relaxation framework.\footnote{See also \cite{ky91} for another technique to evade the difficulty of joint continuity in the original economy.} By definition, $D_\mathcal{R}(t,p)\subset \Gamma(t,p)$ for every $(t,p)\in T\times S^*$. 

\begin{lem}
\label{lem2}
$\Gamma:T\times S^*\twoheadrightarrow \Pi(X)$ is a compact, convex-valued multifunction with $\mathrm{gph}\,\Gamma(\cdot,p)\in \Sigma\otimes\mathrm{Borel}(\Pi(X))$ for every $p\in S^*$ such that $p\mapsto \Gamma(t,p)$ is upper semicontinuous for the weak$^*\!$ topology of $S^*$ for every $t\in T$.
\end{lem}

\begin{proof}
The fact that $\Gamma$ has compact convex values is obvious. To show the upper semicontinuity, fix $t\in T$ arbitrarily and let $\{ p_\alpha \}$ be a net in $S^*$ with $p_\alpha \to p$ weakly$^*\!$ and choose any $P_\alpha\in \Gamma(t,p_\alpha)$ for each $\alpha$ with $P_\alpha\to P$. We need to show that $P\in \Gamma(t,p)$. Suppose, to the contrary, that $P\not\in \Gamma(t,p)$. Then there exists $Q\in \Pi(X)$ such that $Q\,{\succ}_\mathcal{R}(t)\,P$ and $\int\langle p,\imath_X(x) \rangle dQ\le \langle p,\omega(t) \rangle$. It follows from the continuity of ${\succsim}_\mathcal{R}(t)$ and Assumption \ref{assmp2}(ii) that $Q$ is taken such that $Q\,{\succ}_\mathcal{R}(t)\,P$ and $\int\langle p,\imath_X(x) \rangle dQ<\langle p,\omega(t) \rangle$. Thus, for all sufficiently large $\alpha$, we have $Q\,{\succ}_\mathcal{R}(t)\,P$ and $\langle p_\alpha,\int\imath_XdQ \rangle=\int\langle p_\alpha,\imath_X(x) \rangle dQ<\langle p,\omega(t) \rangle$, which contradicts the fact that $P_\alpha\in \Gamma(t,p_\alpha)$. Therefore, $P\in \Gamma(t,p)$. Since 
\begin{align*}
\Gamma(t,p)=\left\{ P\in \Pi(X) \mid \tilde{u}(t,P)\ge \max_{Q\in B_\mathcal{R}(t,p)}\tilde{u}(t,Q) \right\}
\end{align*}
and the marginal function $t\mapsto \max_{Q\in B_\mathcal{R}(t,p)}\tilde{u}(t,Q)$ is measurable by the measurable maximum theorem, we have $\mathrm{gph}\,\Gamma(\cdot,p)\in \Sigma\otimes\mathrm{Borel}(\Pi(X))$ for every $p\in S^*$. 
\end{proof}

\begin{lem}
\label{lem3}
Define the multifunction $I_\Gamma:T\times S^*\twoheadrightarrow E$ by
$$
I_\Gamma(t,p)=\left\{ \int_X\imath_X(x) dP\mid P\in \Gamma(t,p) \right\}.
$$
Then $I_\Gamma$ is a weakly compact, convex-valued multifunction such that its range $I_\Gamma(T\times S^*)$ is bounded and $p\mapsto I_\Gamma(t,p)$ is upper semicontinuous for the weak$^*\!$ topology of $S^*$ and the norm topology of $E$ for every $t\in T$.
\end{lem}

\begin{proof}
It follows from the weak compactness of $X$ that $\sup_{x\in X}\| x \|\le a$ for some $a\ge 0$. Thus, $\sup_{P\in \Gamma(t,p)}\|\int\imath_XdP\|\le a$ for every $(t,p)\in T\times S^*$. Hence, $I_\Gamma(T\times S^*)$ is bounded. Since $\Gamma(t,p)$ is a convex subset of $\Pi(X)$ by Lemma \ref{lem2}, the convexity of $I_\Gamma(t,p)$ is obvious. To show the weak compactness of $\Gamma(t,p)$, fix $(t,p)\in T\times S^*$ arbitrarily and choose a net  $y_\alpha\in I_\Gamma(t,p)$ for each $\alpha$. Then there exists $P_\alpha\in \Gamma(t,p)$ such that $y_\alpha=\int\imath_XdP_\alpha$ for each $\alpha$. Since $\Gamma(t,p)$ is compact by Lemma \ref{lem2}, we can extract a subnet from $\{ P_\alpha \}$ (which we do not relabel) converging to $P\in \Gamma(t,p)$. Hence, the barycenter $\int\imath_XdP$ belongs to $I_\Gamma(t,p)$. It follows from the definition of the topology of the weak convergence of probability measures that for every $x^*\in E^*$, we have 
$$
\langle x^*,y_\alpha \rangle=\int_X\langle x^*,\imath_X(x) \rangle dP_\alpha \to \int_X\langle x^*,\imath_X(x) \rangle dP=\left\langle x^*,\int_X\imath_X(x)dP \right\rangle
$$
because $x\mapsto \langle x^*,\imath_X(x) \rangle$ is a bounded continuous function on $X$. This means that $y_\alpha\to \int\imath_XdP$ weakly in $E$. Thus, $I_\Gamma(t,p)$ is weakly compact. 

To show the upper semicontinuity, fix $t\in T$ arbitrarily and let $\{ p_\alpha \}$ be a net in $S^*$ with $p_\alpha \to p$ weakly$^*\!$ and choose any $y_\alpha\in I_\Gamma(t,p_\alpha)$ for each $\alpha$ with $y_\alpha\to y$ strongly in $E$. We need to show that $y\in I_\Gamma(t,p)$. Suppose, to the contrary, that $y\not\in I_\Gamma(t,p)$. Then for each $\alpha$ there exists $P_\alpha\in \Gamma(t,p_\alpha)$ such that $y_\alpha=\int\imath_XdP_\alpha$. Denote by $\overline{\mathrm{co}}\,X$ be the closed convex hull of $X$. Then the barycenters $\int\imath_XdP_\alpha$ belong to $\overline{\mathrm{co}}\,X$; see \citet[Corollary II.2.8]{du77}. Hence, we have $y\in \overline{\mathrm{co}}\,X$. It follows from Choquet's theorem (see \citet[Proposition 1.2]{ph01}) that there exists $P\in \Pi(X)$ such that $\langle x^*,y \rangle=\int\langle x^*,\imath_X(x) \rangle dP=\langle x^*,\int\imath_XdP \rangle$ for every $x^*\in E^*$. This means that $y=\int\imath_XdP$. In view of $y\not\in I_\Gamma(t,p)$, we have $P\not\in \Gamma(t,p)$. As demonstrated in the proof of Lemma \ref{lem2}, there exists $Q\in \Pi(X)$ such that $Q\,{\succ}_\mathcal{R}(t)\,P$ and $\langle p_\alpha,\int\imath_XdQ \rangle=\int\langle p_\alpha,\imath_X(x) \rangle dQ<\langle p,\omega(t) \rangle$ for all sufficiently large $\alpha$, which contradicts the fact that $P_\alpha\in \Gamma(t,p_\alpha)$. Therefore, $y\in I_\Gamma(t,p)$.
\end{proof}

What is significant in the next lemma is that the upper semicontinuity of $p\mapsto \int I_\Gamma(t,p)d\mu$ is preserved under integration without any assumption on the finite measure space $(T,\Sigma,\mu)$ due to the fact that the upper semicontinuous multifunction $p\mapsto I_\Gamma(t,p)$ has weakly compact convex values. This observation permits us to invoke fixed point theorems in the sequel.

\begin{lem}
\label{lem4}
The Bochner integral $\int I_\Gamma(t,p)d\mu$ of the multifunction $I_\Gamma(\cdot,p):T\twoheadrightarrow E$ is nonempty, weakly compact, and convex for every $p\in S^*$, and the multifunction $p\mapsto \int I_\Gamma(t,p)d\mu$ is upper semicontinuous for the weak$^*\!$ topology of $S^*$ and the norm topology of $E$. 
\end{lem}

\begin{proof}
The nonemptiness and convexity of $\int I_\Gamma(t,p)d\mu$ are obvious because for every $p\in S^*$ the Bochner integrable selectors of $I_\Gamma(\,\cdot,p):T\twoheadrightarrow E$ are precisely of the form $\int\imath_X\lambda(t,dx)\in I_\Gamma(t,p)$ with $\lambda(t)\in \Gamma(t,p)$ a.e.\ $t\in T$ and $\lambda\in \mathcal{R}(T,X)$. The weak compactness of $\int I_\Gamma(t,p)d\mu$ follows from \citet[Theorem 6.1]{ya91}. 

To show the upper semicontinuity, introduce the \textit{support functional} of $C\subset E$ and define $s(\cdot,C):E^*\to \R\cup \{ +\infty \}$ by $s(x^*,C)=\sup_{x\in C}\langle x^*,x \rangle$. Then the weakly compact convex-valued multifunction $p\mapsto \int I_\Gamma(t,p)d\mu$ is upper semicontinuous if and only if $p\mapsto s(x^*,\int I_\Gamma(t,p)d\mu)$ is upper semicontinuous for every $x^*\in E^*$; see \citet[Theorem 17.41]{ab06}. Since $s(x^*,\int I_\Gamma(t,p)d\mu)=\int s(x^*,I_\Gamma(t,p))d\mu$ for every $p\in S^*$ (see \citet[Proposition 8.6.2]{af90}), it suffices to show the upper semicontinuity of $p\mapsto \int s(x^*,I_\Gamma(t,p))d\mu$ for every $x^*\in E^*$. Since $S^*$ is metrizable with respect to the weak$^*\!$ topology, we can resort to sequential convergence in $S^*$. Let $\{ p_n \}$ be a sequence in $S^*$ with $p_n\to p$ weakly$^*\!$. Since the weak compact convex valued multifunction $p\mapsto I_\Gamma(t,p)$ is upper semicontinuous for every $t\in T$ by Lemma \ref{lem3}, the function $p\mapsto s(x^*,I_\Gamma(t,p))$ is upper semicontinuous for every $x^*\in E^*$ and $t\in T$. Fix $x^*\in E^*$ arbitrarily. Since the sequence of functions $t\mapsto s(x^*,I_\Gamma(t,p_n))$ is bounded, we obtain
\begin{align*}
\limsup_{n\to \infty}\int_Ts(x^*,I_\Gamma(t,p_n))d\mu
& \le \int_T \limsup_{n\to \infty}s(x^*,I_\Gamma(t,p_n))d\mu \\
& \le \int_T s(x^*,I_\Gamma(t,p))d\mu
\end{align*}
where the first equality follows from Fatou's lemma and the second inequality exploits the upper semicontinuity of $p\mapsto s(x^*,I_\Gamma(t,p))$. Therefore, $p\mapsto \int s(x^*,I_\Gamma(t,p))d\mu$ is upper semicontinuous for every $x^*\in E^*$.
\end{proof}

A maximal element in $\Pi(X)$ for ${\succsim}_\mathcal{R}(t)$ is called a \textit{satiation point} for ${\succsim}_\mathcal{R}(t)$. Corresponding to the conditions (iii) and (iv) of Assumption \ref{assmp2} on the original preferences, the condition on (non)satiation points for the relaxed preferences takes the following form. 

\begin{lem}
\label{lem5}
\begin{enumerate}[\rm(i)]
\item If $P\in \Pi(X)$ is a satiation point for ${\succsim}_\mathcal{R}(t)$, then $\int\imath_XdP \\\ge \omega(t)$.
\item If $P\in \Pi(X)$ is not a satiation point for ${\succsim}_\mathcal{R}(t)$, then $P$ belongs to the closure of the upper contour set $\{ Q\in \Pi(X)\mid Q\,{\succ}_\mathcal{R}(t)\,P \}$. 
\end{enumerate}
\end{lem}

\begin{proof}
(i): Let $U(t)\subset X$ be the set of satiation points for ${\succsim}(t)$. Given the preference representation \eqref{rp2}, we have $\max_{x\in X}u(t,x)>u(t,y)$ for every $y\in X\setminus U(t)$. We first claim that $P\in \Pi(X)$ is a satiation point for ${\succsim}_\mathcal{R}(t)$ if and only if $P(U(t))=1$. Suppose that $P$ is satiated for ${\succsim}_\mathcal{R}(t)$. We then have $\int u(t,x)dP\ge \int u(t,x)dQ$ for every $Q\in \Pi(X)$. Assume that $P(U(t))<1$. If we choose $Q\in \Pi(X)$ satisfying $Q(U(t))=1$, then $\int u(t,x)dP<\max_{x\in X}u(t,x)=\int u(t,x)dQ$, a contradiction. Conversely, if $P(U(t))=1$, then $\int u(t,x)dP=\max_{x\in X}u(t,x)\ge \int u(t,x)dQ$ for every $Q\in \Pi(X)$. Hence, $P$ is a satiation point for ${\succsim}_\mathcal{R}(t)$. Since $\imath_X(x)-\omega(t)\le 0$ for every $x\in U(t)$ in view of Assumption \ref{assmp2}(iii), integrating this inequality over $U(t)$ with respect to any satiated $P$ yields $\int\imath_XdP-\omega(t)\le 0$. 

(ii): Take any nonsatiation point $P\in \Pi(X)$ for ${\succsim}(t)$. We need to show that there exists a sequence $\{ P_n \}$ in $\Pi(X)$ with $P_n\,{\succ}_\mathcal{R}(t)\,P$ for each $n$ and $P_n\to P$. Since the convex hull of $\Delta(X)$ is dense in $\Pi(X)$ (see \citet[Theorem 15.10]{ab06}) and $\Pi(X)$ is separable, $P$ is approximated arbitrarily by a sequence of the convex combinations of Dirac measures of the form $Q=\sum_{i\in I}\alpha^i\delta_{x^i}\in \Pi(X)$, where $x^i\in X$, $\alpha^i>0$, and $\sum_{i\in I}\alpha_i=1$ with a finite index set $I$. Since $P(U(t))<1$, we may assume without loss of generality that $x^i\in X\setminus U(t)$ for some $i\in I$ whenever $Q$ is close enough to $P$. For each $x^i\in X\setminus U(t)$, choose a sequence $y^i_n\in X$ with $y^i_n\,{\succ}(t)\,x^i$ for each $n$ and $y^i_n\to x^i$ weakly, which is possible by Assumption \ref{assmp2}(iv). Define the probability measure by
$$
P_n=\sum_{\{ i\in I\mid x^i\in U(t) \}}\alpha^i\delta_{x^i}+\sum_{\{ i\in I\mid x^i\in X\setminus U(t) \}}\alpha^i\delta_{y^i_n}.
$$
By construction, we have 
\begin{align*}
\int_Xu(t,x)dP_n
& =\sum_{\{ i\in I\mid x^i\in U(t) \}}\alpha^iu(t,x^i)+\sum_{\{ i\in I\mid x^i\in X\setminus U(t) \}}\alpha^iu(t,y^i_n) \\
& >\sum_{i\in I}\alpha^iu(t,x^i)=\int_Xu(t,x)dQ. 
\end{align*} 
Hence, $P_n\,{\succ}_\mathcal{R}(t)\,Q$ for each $n$. Since $P_n\to Q$ and $Q$ can be taken close arbitrarily to $P$,  we obtain the desired conclusion. 
\end{proof}

\begin{lem}
\label{lem6}
$\int\langle p,\imath_X(x) \rangle dP\ge \langle p,\omega(t) \rangle$ for every $(t,p)\in T\times S^*$ and $P\in \Gamma(t,p)$. 
\end{lem}

\begin{proof}
If $P$ is a satiation point for ${\succsim}_\mathcal{R}(t)$, then by Lemma \ref{lem5}(i), $\int\imath_XdP\ge \omega(t)$, and hence, $\int\langle p,\imath_X(x) \rangle dP=\langle p,\int\imath_XdP \rangle\ge \langle p, \omega(t) \rangle$ a.e.\ $t\in T$. If $P$ is not a satiation point ${\succsim}_\mathcal{R}(t)$ and $\int\langle p,\imath_X(x) \rangle dP<\langle p,\omega(t) \rangle$, it follows from Lemma \ref{lem5}(ii) that there exists $Q\,{\succ}_\mathcal{R}(t)\,P$ such that $\int\langle p,\imath_X(x) \rangle dQ<\langle p,\omega(t) \rangle$, which contradicts the fact that $P\in \Gamma(t,p)$. 
\end{proof}

\begin{proof}[\rm\bf{Proof of Theorem \ref{RWE1}}]
(i): Define the multifunction $\xi:S^*\twoheadrightarrow E$ by
$$
\xi(p)=\int_TI_\Gamma(t,p)d\mu-\int_T\omega(t)d\mu.
$$
Then by Lemma \ref{lem4}, $\xi$ is upper semicontinuous for the weak$^*\!$ topology of $S^*$ and the norm topology of $E$ with weakly compact, convex values. We claim that for every $p\in S^*$ there exists $z\in \xi(p)$ such that $\langle p,z \rangle\le 0$. To this end, fix $p\in S^*$ arbitrarily. By Lemma \ref{lem1}, there exists a measurable function $\lambda_p:T\to \Pi(X)$ such that $\lambda_p(t)\in D_\mathcal{R}(t,p)\subset \Gamma(t,p)$ a.e.\ $t\in T$. Since $\int\imath_X\lambda_p(t,dx)\in I_\Gamma(t,p)$, we have $\iint\imath_X\lambda_p(t,dx)d\mu\in \int I_\Gamma(t,p)d\mu$. Integrating the relaxed budget constraint $\int\langle p,\imath_X(x) \rangle \lambda_p(t,dx)-\langle p,\omega(t) \rangle\le 0$ over $T$ yields $\left\langle p,\iint\imath_X\lambda_p(t,dx)d\mu-\int\omega d\mu \right\rangle\le 0$. Hence, the vector $z=\iint\imath_X\lambda_p(t,dx)d\mu-\int\omega d\mu\in \xi(p)$ satisfies $\langle p,z \rangle\le 0$. 

It follows from the infinite-dimensional version of the Gale--Nikaido Lemma (see \citet[Theorem 3.1]{ya85}) that there exists $p\in S^*$ such that $\xi(p)\cap (-E_+)\ne \emptyset$. Hence, there exists a measurable function $\lambda:T\to \Pi(X)$ with $\lambda(t)\in \Gamma(t,p)$ a.e.\ $t\in T$ satisfying $\iint\imath_X\lambda(t,dx)d\mu\in \int I_\Gamma(t,p)d\mu$ and $\iint\imath_X\lambda(t,dx)d\mu-\int\omega d\mu\le 0$. This means that $\lambda\in \A(\E_\mathcal{R})$ and 
$$
\int_T\left\langle p,\int_X\imath_X(x)\lambda(t,dx) \right\rangle d\mu\le \int_T \langle p,\omega(t) \rangle d\mu.
$$
On the other hand, by Lemma \ref{lem6},
$$
\left\langle p,\int_X\imath_X(x)\lambda(t,dx) \right\rangle\ge \langle p, \omega(t) \rangle \quad\text{a.e.\ $t\in T$}. 
$$
Combining these inequalities yields the equality $\int\langle p,\imath_X(x) \rangle\lambda(t,dx)=\langle p,\omega(t) \rangle$ a.e.\ $t\in T$ for the relaxed budget constraint. Note also that by construction we have
$$
D_\mathcal{R}(t,p)=\Gamma(t,p)\cap \left\{ P\in \Pi(X)\mid \int_X\langle p,\imath_X(x) \rangle dP=\langle p,\omega(t) \rangle \right\}. 
$$
This means that $\lambda(t)$ belongs to $D_\mathcal{R}(t,p)$ a.e. $t\in T$. Therefore, the price-relaxed allocation pair $(p,\lambda)\in S^*\times \mathcal{R}(T,X)$ is a relaxed Walrasian equilibrium for $\E_\mathcal{R}$. 

(ii): This follows from the assertion (i) and Proposition \ref{eqv1}. 
\end{proof}

\subsection{Proof of Theorem \ref{RWE2}}
(i): When $X$ is a finite set, the difficulty of the joint continuity of the valuation functional $(p,x)\mapsto \langle p,x \rangle$ on $S^*\times X$ never arises because the relative topology of $X$ inherited from the weak topology of $E$ is a discrete topology. Indeed, let $(p_\alpha,x_\alpha)$ be a net in $S^*\times X$ converging to $(p,x)\in S^*\times X$. Since $x_\alpha\to x$ weakly means that $x_\alpha=x$ for every $\alpha\ge \alpha_0$ with some $\alpha_0$, we have $\lim_\alpha\langle p_\alpha,x_\alpha \rangle=\lim_\alpha\langle p_\alpha,x \rangle=\langle p,x \rangle$. Hence, the valuation functional is continuous on $S^*\times X$ for the weak$^*$ topology of $E^*$ and the weak topology of $E$. This implies that the relaxed demand multifunction $D_\mathcal{R}(t,\cdot)$ is upper semicontinuous on $S^*$ for every $t\in T$. Let $I_{D_\mathcal{R}}:T\times S^*\to E$ be a multifunction defined by
$$
I_{D_\mathcal{R}}(t,p)=\left\{ \int_X\imath_X(x) dP\mid P\in D_\mathcal{R}(t,p) \right\}. 
$$
Replacing $\Gamma$ by $D_\mathcal{R}$ in the proof of Lemma \ref{lem4} yields that the Bochner integral $\int I_{D_\mathcal{R}}(t,p)d\mu$ is nonempty, weakly compact, and convex for every $p\in S^*$, and the multifunction $p\mapsto \int I_{D_\mathcal{R}}(t,p)d\mu$ is upper semicontinuous for the weak$^*\!$ topology of $S^*$ and the norm topology of $E$. 

Define the multifunction $\xi:S^*\twoheadrightarrow E$ by
$$
\xi(p)=\int_TI_{D_\mathcal{R}}(t,p)d\mu-\int_T\omega(t)d\mu.
$$
Then $\xi$ is upper semicontinuous for the weak$^*\!$ topology of $S^*$ and the norm topology of $E$ with weakly compact, convex values. As in the proof of Proposition \ref{RWE1}, for every $p\in S^*$ there exists $z\in \xi(p)$ such that $\langle p,z \rangle\le 0$. It follows from the infinite-dimensional version of the Gale--Nikaido Lemma (see \citet[Theorem 3.1]{ya85}) that there exists $p\in S^*$ such that $\xi(p)\cap (-E_+)\ne \emptyset$. Hence, there exists a measurable function $\lambda:T\to \Pi(X)$ with $\lambda(t)\in D_\mathcal{R}(t,p)$ a.e.\ $t\in T$ satisfying $\iint\imath_X\lambda(t,dx)d\mu\in \int I_{D_\mathcal{R}}(t,p)d\mu$ and $\iint\imath_X\lambda(t,dx)d\mu-\int\omega d\mu\le 0$. This means that the price-relaxed allocation pair $(p,\lambda)\in S^*\times \mathcal{R}(T,X)$ is a relaxed Walrasian equilibrium for $\E_\mathcal{R}$. 

(ii): This follows from the assertion (i) and Proposition \ref{eqv1}. \qed

\section{Appendix 2}
\label{appdx2}
\subsection{Proof of Theorem \ref{WE3}}
Given the duality $L^\infty(\nu)^*=\mathit{ba}(\nu)$, denote by $\sigma(\mathit{ba},L^\infty)$ the weak$^*$ topology of $\mathit{ba}(\nu)$. Define the normalized price space by $S^*=\{ \pi\in \mathit{ba}_+(\nu)\mid \langle \pi,\psi \rangle=1 \}$, where $\psi$ is taken from the norm interior of $L^\infty_+(\nu)$. Then $S^*$ is $\sigma(\mathit{ba},L^\infty)$-compact (i.e., weakly$^*$ compact) and convex. 

\begin{lem}
\label{lem7}
There exists a sequence $\{ \varphi_n \}$ in $L^1(\nu)$ such that $\mathit{ba}(\nu)=\overline{\{ \varphi_n \}}^{\,\sigma(\mathit{ba},L^\infty)}$.
\end{lem}

\begin{proof}
Consider the natural embedding $L^1(\nu)\subset L^1(\nu)^{**}=\mathit{ba}(\nu)$ and note that $L^1(\nu)$ is $\sigma(\mathit{ba},L^\infty)$-dense subset of $\mathit{ba}(\nu)$; see \citet[Corollary V.4.6]{ds58}. Since $L^1(\nu)$ is separable in view of the countable generation of $\F$, there exists a countable dense set $\{ \varphi_n \}$ of  $L^1(\nu)$ with respect to the norm topology. Since $L^1(\nu)=\overline{\{ \varphi_n \}}^{\,\| \cdot \|}\subset \overline{\{ \varphi_n \}}^{\,\sigma(\mathit{ba},L^\infty)}$, where $\overline{\{ \varphi_n \}}^{\,\| \cdot \|}$ is the norm closure of $\{ \varphi_n \}$ in $L^1(\nu)$ and $\overline{\{ \varphi_n \}}^{\,\sigma(\mathit{ba},L^\infty)}$ is the weak$^*$ closure of $\{ \varphi_n \}$ in $\mathit{ba}(\nu)$, we have $\mathit{ba}(\nu)=\overline{L^1(\nu)}^{\,\sigma(\mathit{ba},L^\infty)}\subset \overline{\{ \varphi_n \}}^{\,\sigma(\mathit{ba},L^\infty)}$. Hence, $\mathit{ba}(\nu)=\overline{\{ \varphi_n \}}^{\,\sigma(\mathit{ba},L^\infty)}$.
\end{proof}

Given the technique explored in Subsection \ref{subsec2}, the existence of equilibrium prices in $\mathit{ba}_+(\nu)$ is more or less routine because it is again a direct application of the Gale--Nikaido lemma in $L^\infty(\nu)$. A key result is Theorem \ref{RWE3} below, for which we outline the proof. 

\begin{obsv}
\label{obsv1}
Lemma \ref{lem1} holds as it stands for $E=L^\infty(\nu)$ and $X\subset L^\infty_+(\nu)$.
\end{obsv}

\begin{obsv}
To see the validity of Lemma \ref{lem2}, it suffices to show that $\langle \pi,\int\imath_XdP \rangle=\int_X\langle \pi,\imath_X(x) \rangle dP$ for every $P\in \Pi(X)$ and $\pi\in \mathit{ba}(\nu)$, where $\int\imath_XdP\in L^\infty(\nu)$ is the Gelfand integral of $\imath_X$ with respect to $P$. By Lemma \ref{lem7}, for every $\pi\in \mathit{ba}(\nu)$ there exists a sequence $\varphi_n\in L^1(\nu)$ such that $\varphi_n\to \pi$ for $\sigma(\mathit{ba},L^\infty)$-topology. We then have $\langle \varphi_n,\int\imath_XdP \rangle=\int \langle \varphi_n,\imath_X(x) \rangle dP$ for each $n$. Taking the limit in the both side of this equality yields
\begin{align*}
\left\langle \pi,\int_X\imath_X(x)dP \right\rangle
& =\lim_{n\to \infty}\left\langle \varphi_n,\int_X\imath_X(x)dP \right\rangle=\lim_{n\to \infty}\int_X\langle \varphi_n,\imath_X(x)\rangle dP \\
& =\int_X\lim_{n\to \infty}\langle \varphi_n,\imath_X(x) \rangle dP=\int_X\langle \pi,\imath_X(x) \rangle dP
\end{align*}
where the third equality in the above employs the Lebesgue dominated convergence theorem in view of the boundedness of $X\subset L^\infty_+(\nu)$. 
\end{obsv}

Now Lemma \ref{lem3} takes the following form. 
\begin{lem}
\label{lem8}
Define the multifunction $I_\Gamma:T\times S^*\twoheadrightarrow L^\infty(\nu)$ by
$$
I_\Gamma(t,\pi)=\left\{ \int_X\imath_X(x) dP\mid P\in \Gamma(t,\pi) \right\}.
$$
Then $I_\Gamma$ is a weakly$^*$ compact, convex-valued multifunction such that its range $I_\Gamma(T\times S^*)$ is bounded and $\pi\mapsto I_\Gamma(t,\pi)$ is upper semicontinuous for the weak$^*\!$ topology of $S^*$ and the norm topology of $L^\infty(\nu)$ for every $t\in T$.
\end{lem}

\begin{proof}
It follows from the weak$^*$ compactness of $X$ that $\sup_{x\in X}\| x \|\le a$ for some $a\ge 0$. Thus, $\sup_{P\in \Gamma(t,\pi)}\|\int\imath_XdP\|\le a$ for every $(t,\pi)\in T\times S^*$. Hence, $I_\Gamma(T\times S^*)$ is bounded, and hence, it is weakly$^*\!$ relatively compact. Since $\Gamma(t,\pi)$ is a convex subset of $\Pi(X)$ by Lemma \ref{lem2}, the convexity of $I_\Gamma(t,\pi)$ is obvious. To show the weak$^*$ compactness of $\Gamma(t,\pi)$, it suffices to show that $\Gamma(t,\pi)$ is weakly$^*\!$ closed. To this end, fix $(t,\pi)\in T\times S^*$ arbitrarily and choose a net $\psi_\alpha\in I_\Gamma(t,\pi)$ for each $\alpha$. Then there exists $P_\alpha\in \Gamma(t,\pi)$ such that $\psi_\alpha=\int\imath_XdP_\alpha$ for each $\alpha$. Since $\Gamma(t,\pi)$ is compact by Lemma \ref{lem2}, we can extract a subnet from $\{ P_\alpha \}$ (which we do not relabel) converging to $P\in \Gamma(t,\pi)$. Hence, the barycenter $\int\imath_XdP$ belongs to $I_\Gamma(t,\pi)$. It follows from the definition of the topology of the weak convergence of probability measures that for every $\varphi\in L^1(\nu)$, we have 
$$
\langle \varphi,\psi_\alpha \rangle=\int_X\langle \varphi,\imath_X(x) \rangle dP_\alpha \to \int_X\langle \varphi,\imath_X(x) \rangle dP=\left\langle \varphi,\int_X\imath_X(x)dP \right\rangle
$$
because $x\mapsto \langle \varphi,\imath_X(x) \rangle$ is a bounded continuous function on $X$. This means that $\psi_\alpha\to \int\imath_XdP$ weakly$^*\!$ in $L^\infty(\nu)$. Thus, $I_\Gamma(t,p)$ is weakly$^*\!$ closed. 

To show the upper semicontinuity, fix $t\in T$ arbitrarily and let $\{ \pi_\alpha \}$ be a net in $S^*$ with $\pi_\alpha \to \pi$ weakly$^*\!$ and choose any $\psi_\alpha\in I_\Gamma(t,\pi_\alpha)$ for each $\alpha$ with $\psi_\alpha\to \psi$ strongly in $L^\infty(\nu)$. We need to show that $\psi\in I_\Gamma(t,\pi)$. Suppose, to the contrary, that $\psi\not\in I_\Gamma(t,\pi)$. Then for each $\alpha$ there exists $P_\alpha\in \Gamma(t,\pi_\alpha)$ such that $\psi_\alpha=\int\imath_XdP_\alpha$. Denote by $\overline{\mathrm{co}}^\mathit{\,w^*}\!X$ be the weakly$^*\!$ closed convex hull of $X$. Then the barycenters $\int\imath_XdP_\alpha$ belong to $\overline{\mathrm{co}}^\mathit{\,w^*}\!X$; see \citet[Lemma 3.1]{sa16} and \citet[Lemma 2.1]{ks14b}. Hence, we have $\psi\in \overline{\mathrm{co}}^\mathit{\,w^*}\!X$. It follows from Choquet's theorem (see \citet[Proposition 1.2]{ph01}) that there exists $P\in \Pi(X)$ such that $\langle \varphi,\psi \rangle=\int\langle \varphi,\imath_X(x) \rangle dP=\langle \varphi,\int\imath_XdP \rangle$ for every $\varphi\in L^1(\nu)$. This means that $\psi=\int\imath_XdP$. In view of $\psi\not\in I_\Gamma(t,\pi)$, we have $P\not\in \Gamma(t,\pi)$. As demonstrated in the proof of Lemma \ref{lem2}, there exists $Q\in \Pi(X)$ such that $Q\,{\succ}_\mathcal{R}(t)\,P$ and $\langle \pi_\alpha,\int\imath_XdQ \rangle=\int\langle \pi_\alpha,\imath_X(x) \rangle dQ<\langle \pi,\omega(t) \rangle$ for all sufficiently large $\alpha$, which contradicts the fact that $P_\alpha\in \Gamma(t,\pi_\alpha)$. Therefore, $\psi\in I_\Gamma(t,\pi)$.
\end{proof}

Corresponding to Lemma \ref{lem4}, we obtain the following. 

\begin{lem}
\label{lem9}
The Gelfand integral $\int I_\Gamma(t,\pi)d\mu$ of the multifunction $I_\Gamma:T\times S^*\twoheadrightarrow L^\infty(\nu)$ is nonempty, weakly$^*\!$ compact, and convex for every $\pi\in S^*$, and the multifunction $\pi\mapsto \int I_\Gamma(t,\pi)d\mu$ is upper semicontinuous for the weak$^*\!$ topology of $S^*$ and the norm topology of $L^\infty(\nu)$. 
\end{lem}

\begin{proof}
The nonemptiness and convexity of $\int I_\Gamma(t,\pi)d\mu$ are obvious because for every $\pi\in S^*$ the Gelfand integrable selectors of $I_\Gamma(\,\cdot,\pi):T\twoheadrightarrow L^\infty(\nu)$ are precisely of the form $\int\imath_X\lambda(t,dx)\in I_\Gamma(t,\pi)$ with $\lambda(t)\in \Gamma(t,\pi)$ a.e.\ $t\in T$ and $\lambda\in \mathcal{R}(T,X)$. 

To show the weak$^*\!$ compactness of $\int I_\Gamma(t,\pi)d\mu$, introduce the \textit{support functional} of $C\subset L^\infty(\nu)$ and define $s(\cdot,C):L^1(\nu)\to \R\cup \{ +\infty \}$ by $s(\varphi,C)=\sup_{\psi\in C}\langle \varphi,\psi \rangle$. Since $t\mapsto I_\Gamma(t,\pi)$ is an integrably bounded multifunction with weakly$^*\!$ compact, convex values by Lemma \ref{lem8}, it suffices to show that $t\mapsto I_\Gamma(t,\pi)$ is \textit{weakly$^*\!$ scalarly measurable} in the sense that $t\mapsto s(\varphi,I_\Gamma(t,\pi))$ is measurable for every $\psi\in L^1(\nu)$; see \citet[Claim 2 to the proof of Theorem 2]{kh85} or \citet[Proposition 2.3(i) and Theorem 4.5]{ckr11}. To this end, it suffices to show that $t\mapsto \sup_{P\in \Gamma(t,\pi)}\int \langle \varphi,\imath_X(x) \rangle dP$ is measurable for every $\varphi\in L^1(\nu)$ because 
\begin{align*}
s(\varphi,I_\Gamma(t,\pi))=\sup_{P\in \Gamma(t,\pi)}\left\langle \varphi,\int_X\imath_X(x)dP \right\rangle=\sup_{P\in \Gamma(t,\pi)}\int_X\langle \varphi,\imath_X(x)\rangle dP
\end{align*} 
in view of the Gelfand integrability of $\imath_X$. Since $\mathrm{gph}\,\Gamma(\cdot,\pi)\in \Sigma\otimes \mathrm{Borel}(\Pi(X))$ by Lemma \ref{lem2} and $P\mapsto \int \langle \varphi,\imath_X(x) \rangle dP$ is continuous because $x\mapsto \langle\varphi,\imath_X(x) \rangle$ is a bounded continuous function on $X$, it follows from the measurable maximum theorem that the marginal function $t\mapsto \sup_{P\in \Gamma(t,\pi)}\int \langle \varphi,\imath_X(x) \rangle dP$ is measurable. 

The weakly$^*\!$ compact convex-valued multifunction $\pi\mapsto \int I_\Gamma(t,\pi)d\mu$ is upper semicontinuous if and only if $\pi\mapsto s(\varphi,\int I_\Gamma(t,\pi)d\mu)$ is upper semicontinuous for every $\varphi\in L^1(\nu)$; see \citet[Theorem 17.41]{ab06}. Since $s(\varphi,\int I_\Gamma(t,\pi)d\mu)=\int s(\varphi,I_\Gamma(t,\pi))d\mu$ for every $\pi\in S^*$ (see \citet[Proposition 2.3(i) and Theorem 4.5]{ckr11}), it suffices to show the upper semicontinuity of $\pi\mapsto \int s(\varphi,I_\Gamma(t,\pi))d\mu$ for every $\varphi\in L^1(\nu)$. The rest of the proof is same with the proof of Lemma \ref{lem4}. 
\end{proof}

\begin{obsv}
\label{obsv2}
Lemma \ref{lem5} holds as it stands. Lemma \ref{lem6} holds by the same reason with Observation \ref{obsv1}. 
\end{obsv}

An analogue of the first assertion of Theorem \ref{RWE1} with the commodity space of $L^\infty(\nu)$ with the Gelfand integral setting is provided as follows.

\begin{thm}
\label{RWE3}
Let $(T,\Sigma,\mu)$ be a finite measure space and $(\Omega,\F,\nu)$ be a countably generated $\sigma$-finite measure space. Then for every economy $\E^G$ satisfying Assumptions \ref{assmp4} and \ref{assmp5}, there exists a relaxed Walrasian equilibrium with free disposal for $\E^G_\mathcal{R}$ with a positive price. 
\end{thm}

\begin{proof}
Define the multifunction $\xi:S^*\twoheadrightarrow L^\infty(\nu)$ by
$$
\xi(\pi)=\int_TI_\Gamma(t,\pi)d\mu-\int_T\omega(t)d\mu.
$$
Then by Lemma \ref{lem9}, $\xi$ is upper semicontinuous for the weak$^*\!$ topology of $S^*$ and the norm topology of $L^\infty(\nu)$ with weakly$^*\!$ compact, convex values. As in the proof of Theorem \ref{RWE1}(i), we can show that for every $\pi\in S^*$ there exists $z\in \xi(\pi)$ such that $\langle \pi,z \rangle\le 0$. Hence, it follows from the infinite-dimensional version of the Gale--Nikaido Lemma (see \citet[Theorem 3.1]{ya85}) that there exists $\pi\in S^*$ such that $\xi(\pi)\cap (-L^\infty_+(\nu))\ne \emptyset$. The rest of the proof is same with the proof of Theorem \ref{RWE1}(i) with replacing the Bochner integrals by Gelfand ones and the use of Lemma \ref{lem6} with Observation \ref{obsv2}
\end{proof}

\begin{proof}[Proof of Theorem \ref{WE3}]
(i): Let $(\pi,\lambda)\in \mathit{ba}_+(\nu)\times \A(\E^G_\mathcal{R})$ be a relaxed  Walrasian equilibrium with free disposal for $\E^G_\mathcal{R}$ assured in Theorem \ref{RWE3}. By the Yosida--\hspace{0pt}Hewitt decomposition of finitely additive measures (see \citet[Theorems 1.22 and 1.24]{yh52}), $\pi$ is decomposed uniquely into $\pi=\pi_1+\pi_2$, where $\pi_1\ge 0$ is countably additive and $\pi_2\ge 0$ is purely finitely additive. (Here, $\pi_2$ is \textit{purely finitely additive} if every countably additive measure $\pi'$ on $\F$ satisfying $0\le \pi'\le \pi_2$ is identically zero.) Furthermore, there exists a sequence $\{ \Omega_n \}$ in $\F$ such that (a) $\Omega_n\subset \Omega_{n+1}$ for each $n=1,2,\dots$; (b) $\lim_n\pi_1(\Omega\setminus \Omega_n)=0$; (c) $\pi_2(\Omega_n)=0$ for each $n=1,2,\dots$. 

We claim that $(\pi_1,\lambda)$ a relaxed Walrasian equilibrium with free disposal for $\E^G_\mathcal{R}$. To this end, suppose that $P\,{\succ}_\mathcal{R}(t)\,\lambda(t)$. We need to demonstrate that $\int\langle \pi_1,\imath_X(x)dP \rangle>\langle \pi_1,\omega(t) \rangle$. It follows from the definition of relaxed Walrasian equilibria that $\int\langle \pi,\imath_X(x)dP \rangle>\langle \pi,\omega(t) \rangle$. Define $X_n=\{ \psi\in X\mid \psi(s)=0 \ \forall s\in \Omega\setminus \Omega_n \}$. Then $X_n\subset X_{n+1}$ for each $n$ by virtue of condition (a) and $P(\bigcup X_n)=P(X)=1$. Without loss of generality we may assume that $P(X_n)>0$ for each $n$. Let $P_n\in \Pi(X)$ be the conditional probability measure of $X_n$ defined by $P_n(Z)=P(Z\cap X_n)/P(X_n)$ with $Z\in \mathrm{Borel}(X,\mathit{w}^*)$, where the relevant Borel $\sigma$-algebra of $X\subset L^\infty(\nu)$ is with respect to the weak$^*$ topology of $L^\infty(\nu)$. By construction, $P_n(X_n)=1$ for each $n$. Since each $P_n$ is absolutely continuous with respect to $P$, there is a Radon--Nikodym derivative $w_n\in L^1(P)$ of $P_n$. Since $P_n(Z)\to P(Z)$ for every $Z\in \mathrm{Borel}(X,\mathit{w}^*)$ by condition (b), it is easy to see that $w_n\to \chi_X$ strongly in $L^1(P)$. Choose any continuous function $v$ on $X$. It follows from the Lebesgue dominated convergence theorem that $\int vdP_n=\int vw_ndP\to \int vdP$, and hence, $P_n\to P$ in $\Pi(X)$. By the continuity of ${\succsim}_\mathcal{R}(t)$, we have $P_n\,{\succ}_\mathcal{R}(t)\,\lambda(t)$ and $\int\langle \pi,\imath(x) \rangle dP_n>\langle \pi,\omega(t) \rangle$ for all sufficiently large $n$. Since $X_n$ is closed and convex, and $P_n(X_n)=1$, we have $\int\imath_XdP_n\in X_n$ by \citet[Lemma 3.1]{sa16}. Let $\psi_n:=\int\imath_XdP_n$. Since $\langle \pi_2,\psi_n \rangle=\int\psi_nd\pi_2=\int_{\Omega_n}\psi_nd\pi_2=0$ by condition (c), we have $\langle \pi,\psi_n \rangle=\langle \pi_1,\psi_n \rangle+\langle \pi_2,\psi_n \rangle=\langle \pi_1,\psi_n \rangle$. In view of $\int\imath_XdP\ge \int\imath_XdP_n$, we obtain
\begin{align*}
\int_X\langle \pi_1,\imath_X(x) \rangle dP=\left\langle \pi_1,\int_X\imath_X(x)dP \right\rangle\ge \left\langle \pi_1,\psi_n \right\rangle=\langle \pi,\psi_n \rangle
& >\langle \pi,\omega(t)\rangle \\
& \ge \langle \pi_1,\omega(t)\rangle
\end{align*} 
as desired. This also implies that $\pi_1\ne 0$. Since $\pi$ is absolutely continuous with respect to $\nu$, the Radon Nikodym derivative of $\pi_1$ is an equilibrium price in $L^1(\nu)$. 

(ii): This follows from the assertion (i) and Proposition \ref{eqv2}. 
\end{proof}

\end{document}